\documentclass[11pt]{article}

\usepackage[margin=1in]{geometry}

\usepackage{lmodern}
\usepackage[T1]{fontenc}

\usepackage[english]{babel} 
\usepackage[utf8]{inputenc}

\usepackage{mathrsfs} 
\usepackage{amsfonts}
\usepackage{marvosym} 
\usepackage{amsmath,amssymb,amsthm}

\usepackage[usenames, dvipsnames, svgnames]{xcolor}
\definecolor{vert}{RGB}{15,120,5}
\definecolor{gris}{RGB}{128,128,128}
\definecolor{bleu}{RGB}{0,50,150}
\definecolor{rouge}{RGB}{149,24,24}
\usepackage[colorlinks = true,linkcolor=bleu,citecolor=vert]{hyperref}

\usepackage[nameinlink]{cleveref} 
\crefname{equation}{}{}

\usepackage{bm}
\usepackage{bbm}

\usepackage{epsfig}

\usepackage{tikz-cd} 

\usepackage{stmaryrd}
\usepackage{manfnt}
\usepackage{multicol}
\usepackage{array}
\usepackage{multirow}

\usepackage{fancyhdr} 
\usepackage{fancybox} 

\newcommand{\titre}{}
\newcommand{\auteur}{}

\usepackage{calligra}
\title{\titre}
\author{\auteur}

\usepackage{tikz}
\usetikzlibrary{calc, intersections}

\usepackage{adjustbox}

\numberwithin{equation}{subsubsection}

\usepackage{thmtools}
\theoremstyle{plain}
\newtheorem{thm}{Theorem}[subsection]
\newtheorem{prop}[thm]{Proposition}

\newtheorem*{thm*}{Theorem}

\newtheorem{nota}[thm]{Notation}

\newtheorem{lem}[thm]{Lemma}
\newtheorem{cor}[thm]{Corollary} 

\theoremstyle{definition}
\newtheorem{defi}[thm]{Definition}
\newtheorem{constr}[thm]{Construction}
\theoremstyle{remark}

\newtheorem{rem}[thm]{Remark}

\numberwithin{equation}{thm}

\newcommand{\C}{\mathbb{C}}

\newcommand{\Q}{\mathbb{Q}}
\newcommand{\Z}{\mathbb{Z}}
\newcommand{\N}{\mathbb{N}}

\newcommand{\Spec}{\operatorname{Spec}}
\newcommand{\A}{\mathbb{A}}

\newcommand{\HH}{\mathrm{H}}
\newcommand{\D}{\mathrm{D}}

\newcommand{\ccal}{\mathcal{C}}

\newcommand{\dcal}{\mathcal{D}}
\newcommand{\mcal}{\mathcal{M}}

\newcommand{\Dd}{\mathrm{D}}
\newcommand{\Ind}{\mathrm{Ind}}

\newcommand{\colim}{\mathrm{colim}}

\newcommand{\Hom}{\mathrm{Hom}}

\newcommand{\Gm}{{\mathbb{G}_m}}

\newcommand{\hcal}{\mathcal{H}}

\DeclareMathOperator{\sHom}{\mathscr{H}\text{\kern -3pt {\calligra\large om}}\,}

\newcommand{\Map}{\mathrm{Map}}

\newcommand{\HHp}{{\ ^{\mathrm{p}}\HH}}
\newcommand{\HHh}{{\ ^{\mathrm{std}}\HH}}

\newcommand{\Sch}{\mathrm{Sch}}
\newcommand{\op}{\mathrm{op}}

\newcommand{\catinfty}{\mathrm{Cat}_\infty}

\newcommand*{\Xfrak}{\mathfrak{X}}
\newcommand*{\Yfrak}{\mathfrak{Y}}
\newcommand*{\Zfrak}{\mathfrak{Z}}
\newcommand*{\Ufrak}{\mathfrak{U}}

\DeclareFontFamily{U}{BOONDOX-calo}{\skewchar\font=45 }
\DeclareFontShape{U}{BOONDOX-calo}{m}{n}{
  <-> s*[1.05] BOONDOX-r-calo}{}
\DeclareFontShape{U}{BOONDOX-calo}{b}{n}{
  <-> s*[1.05] BOONDOX-b-calo}{}
\DeclareMathAlphabet{\mathcalboondox}{U}{BOONDOX-calo}{m}{n}
\SetMathAlphabet{\mathcalboondox}{bold}{U}{BOONDOX-calo}{b}{n}
\DeclareMathAlphabet{\mathbcalboondox}{U}{BOONDOX-calo}{b}{n}

\newcommand{\Shv}{\operatorname{\mathrm{Sh}}}

\newcommand{\Nilp}{\operatorname{\mathrm{Nilp}}}

\newcommand{\DM}{\operatorname{\mathrm{DM}_\et}}

\newcommand{\et}{{\mathrm{\'et}}}

\usepackage{appendix}

\begin{document}
\title{Mixed Hodge modules on stacks}
\author{Swann Tubach}
\date{}
\maketitle

\begin{abstract}
  Using the $\infty$-categorical enhancement of mixed Hodge modules constructed by the author in a previous paper, we explain how mixed Hodge modules canonically extend to algebraic stacks, together with all the $6$ operations and weights. We also prove that Drew's approach to motivic Hodge modules gives an $\infty$-category that embeds fully faithfully in mixed Hodge modules, and we identify the image as mixed Hodge modules of geometric origin.
\end{abstract}
 \tableofcontents

\subsection*{Introduction}
Let $X$ be a complex algebraic variety and $n\in \N$ be an integer. Deligne's work in \cite{MR0498552} gives a polarisable mixed Hodge structure on the singular cohomology $\HH^n_\mathrm{sing}(X(\C),\Q)$ of the complex points of $X$, seen as an analytic variety. M. Saito's category of algebraic mixed Hodge modules (\cite{MR1047415}) on $X$ is an Abelian category $\mathrm{MHM}(X)$ modelled on perverse sheaves which is a relative version of polarisable mixed Hodge structures. Its derived category $\Dd^b(\mathrm{MHM}(X))$ is endowed with the $6$ sheaf operations and any complex of mixed Hodge modules $K$ has an underlying complex $\mathrm{rat}(K)$ of perverse sheaves. The mixed Hodge structure on $\HH^n_\mathrm{sing}(X(\C),\Q)$ can be recovered as $\HH^n(f_*\Q_X)$ where $$f_*\colon \Dd^b(\mathrm{MHM}(X))\to\Dd^b(\mathrm{MHM}(\Spec(\C)))\simeq \Dd^b(\mathrm{MHS}_\C)$$ is the pushforward of the map $f:X\to\Spec(\C)$ and $\Q_X\in\Dd^b(\mathrm{MHM}(X))$ is the unit for the tensor product, whose underlying sheaf of perverse sheaves is the constant sheaf of $\Q$-vector spaces $\Q_X\in\Shv(X(\C),\Q)$ under Beilinson's (\cite{MR0923133}) isomorphism $\Dd^b(\mathrm{Perv}(X,\Q))\simeq \Dd^b_c(X(\C)^\mathrm{an},\Q)$. This point of view is very powerful as the formalism of the 6 operations is useful to make computations and reductions.

If $\Xfrak$ is a reasonable algebraic stack (say a global quotient stack or a stack exhausted étale locally by such), one can construct a mixed Hodge structure on $\HH^n_\mathrm{sing}(\Xfrak(\C),\Q)$ either by hyperdescent as in \cite[Section 6.1]{MR0498552} or by using exhaustions as in \cite[Section 5.2]{davisonPurity2CalabiYauCategories2024}. One purpose of this article (see \Cref{sectionStacks}) is to give an extension of M. Saito's derived category, together with the $6$ operations, to algebraic stacks so that the above mixed Hodge structures can be recovered as $\HH^n(f_*\Q_\Xfrak)$ with $f:\Xfrak\to\Spec(\C)$ the structural map:

\begin{thm*}[\Cref{opera}, \Cref{wstruct} and \Cref{nearbcycles}]
There exists a canonical extension of the derived category of mixed Hodge modules to algebraic stacks over the complex numbers. It has the $6$-operations, nearby cycles, and a notion of weights. Over stacks with affine stabilisers this notion of weights gives rise to a weight structure \emph{à la} Bondarko.
\end{thm*}
The proof relies on the $\infty$-categorical enhancement of mixed Hodge modules obtained in \cite{SwannRealisation}, and then on Liu and Zheng's work (\cite{liuEnhancedSixOperations2017}) on extension of formalisms of $6$ operations to stacks. The construction of nearby cycles is based on a motivic construction of the unipotent nearby cycles functor considered in \cite{cassCentralMotivesParahoric2024b}, who give a natural setting in which Ayoub's construction of nearby cycles (\cite{MR2438151}) works. We hope that this paper will be useful as a toolbox for studying the Hodge cohomology of stacks.  It has been used by T. Kinjo in \cite{kinjoDecompositionTheoremGood2024} to prove purity statements and decomposition theorems for the homology of stacks having a good moduli space. \\

There is another approach to giving a relative version of mixed Hodge structures. In \cite{drewMotivicHodgeModules2018}, Drew constructs an $\infty$-category of \emph{motivic Hodge modules} $\mathrm{DH}(X)$ which is endowed with the $6$ operations as well as with a notion of weights. If $X=\Spec\C$ this category embeds fully faithfully in $\Dd(\Ind\mathrm{MHM}^p_\C)$ the derived category of the indization of mixed Hodge modules. This construction has the advantage of being quite straightforward: one consider the commutative algebra $\hcal$ in the $\infty$-category of Voevodsky étale motives that represents Hodge cohomology, and then one takes modules over this algebra. In comparison, M. Saito's construction is very delicate and requires a lot of attention in order to work. The major drawback of Drew's construction is that there was no easy construction of a t-structure, hence one looses access to an abelian category. The second purpose of this article is to prove that Drew's category gives the right thing: it is endowed with a t-structure and it embeds fully faithfully in the derived category of mixed Hodge modules. We also identify its image.

\begin{thm*}[\Cref{drewcomp}]
  Let $X$ be a finite type scheme over the complex numbers. The $\infty$-category of motivic Hodge modules on $X$ of Drew embeds in the derived category of ind-mixed Hodge modules on $X$. Its image is the category generated under shifts, and colimits by objects of the form $f_*\Q_Y$ with $f:Y\to X$ a proper morphism.
\end{thm*}

There is also an improved version of this theorem with enriched motivic Hodge modules, that reach the $\infty$-category generated under shifts and colimits by objects of the form $f_*g^*H$ with $f:Y\to X$ proper, $g:Y\to\Spec\C$ the structural morphism and $H\in\mathrm{MHS}_\C$ a graded polarisable rational mixed Hodge structure over $\C$. We refer the reader to \Cref{sectionMHMmotivic} for more details.
To prove this result, we use that by the work of Drew and Gallauer (\cite{MR4560376}) the $\infty$-categorical enhancement of mixed Hodge modules provides a realisation functor  $$\rho_{\mathrm{H}}\colon\DM\to\Ind\Dd^b(\mathrm{MHM}(-))$$ from the presentable $\infty$-category of rational étale Voevodsky motives to the indization of the derived category of mixed Hodge modules, that commutes with the operations and is colimit preserving. By abstract nonsense this functor will factor through Drew's category of mixed Hodge modules, and gives a fully faithful functor $\mathrm{DH}\to\Ind\Dd^b(\mathrm{MHM})$. The identification of the image is inspired by the proof of Ayoub in the Betti case (\cite[Theorem 1.98]{ayoubAnabelianPresentationMotivic2022}), and relies on the semi-simplicity of smooth and proper pushforwards of pure mixed Hodge modules.

\subsubsection*{Organisation of the paper}
In the first \Cref{recollections} of this article we recall how to construct the $\infty$-categorical enhancement of the $6$ operations for mixed Hodge modules, and why this gives a Hodge realisation of étale motives. In the second \Cref{sectionMHMmotivic} we show that Drew's construction embeds fully faithfully in mixed Hodge modules. In the last \Cref{sectionStacks} we explain how to extend mixed Hodge modules on stacks, finishing with a comparison to existing constructions.

\subsection{Recollections on $\infty$-categorical enhancements of mixed Hodge modules}
\label{recollections}
\subsubsection{Construction of the enhancement}
In \cite{SwannRealisation}, we proved that Saito's construction of the triangulated bounded derived category of algebraic mixed Hodge modules, together with the $6$ operations, can be enhanced to the world of $\infty$-categories. Let us recall how this works:

For every separated finite type $\C$-scheme $X$, the bounded derived category $\Dd^b(\mathrm{MHM}(X))$ of mixed Hodge modules carries a natural \emph{standard} t-structure (called the \emph{constructible} t-structure in \emph{loc. cit.} but we prefer the name standard because all complexes on $\Dd^b(\mathrm{MHM}(X))$ are constructible) which is characterised by the fact that 
$$\Dd^b(\mathrm{MHM}(X))^{t_\mathrm{std}\in [a,b]} = \{K\in\Dd^b(\mathrm{MHM}(X))\mid \forall x\in X\text{ closed}, \ x^*K\in \Dd^{[a,b]}(\mathrm{MHS}_\C)\}$$
for every $-\infty\leqslant a\leqslant b\leqslant +\infty$. If we endow $\Dd_c^b(X(\C)^\mathrm{an},\Q)$ with its canonical t-structure induced by the inclusion $$\Dd_c^b(X(\C)^\mathrm{an},\Q)\subset \Dd(\Shv(X(\C)^\mathrm{an},\Q))$$ then the "underlying $\Q$-structure functor" 
$$\mathrm{rat}\colon\Dd^b(\mathrm{MHM}(X))\to\Dd^b_c(X(\C)^\mathrm{an},\Q)$$ is t-exact if the left-hand side is endowed with the standard t-structure. Denote by $\mathrm{MHM}_\mathrm{std}(X)$ the heart of the standard t-structure. The crucial result in \cite{SwannRealisation} is the following, whose proof is adapted from Nori's proof (\cite{MR1940678}) of the analogous result for $\Dd^b_c(X(\C)^\mathrm{an},\Q)$:
\begin{thm}[{\cite[Corollary 2.19]{SwannRealisation}}]
  The canonical functor 
  $$\Dd^b(\mathrm{MHM}_\mathrm{std}(X))\to\Dd^b(\mathrm{MHM}(X))$$ is an equivalence.
\end{thm}
Using this, because pullbacks $f^*$ are t-exact for the standard t-structure, it is not hard to see that they are derived functors, hence that they canonically have a $\infty$-categorical enhancement which is functorial in $f$, and the same can be said for the tensor product. The enhancements of the other operations arises by adjunction: if a $\infty$-functor between stable $\infty$-categories has an adjoint on the homotopy triangulated category, it has an $\infty$-categorical adjoint by \cite[Theorem 3.3.1]{MR4093970}.

In particular, we have a functor 
$$\Dd^b(\mathrm{MHM}(-)):\Sch_\C^\op\to\mathrm{CAlg}(\mathrm{St})$$ taking values in the $\infty$-category of stably symmetric monoidal $\infty$-categories and symmetric monoidal exact functors. It sends a map $f:Y\to X$ of finite type $\C$-schemes (say, separated) to an $\infty$-functor $f^*$ lifting Saito's pullback functor on the homotopy category. For smooth $f$, the functor $f^*$ admits a left adjoint $f_\sharp$, and moreover $\Dd^b(\mathrm{MHM}(-))$ satisfies the usual axioms considered in 
motivic context: $\A^1$-invariance, $\mathbb{P}^1$-stability, smooth base change, proper base change \emph{etc}. In particular, the underlying homotopy functor is a \emph{motivic triangulated category} in the sense of \cite[Definition 2.4.45]{MR3971240}, it satisfies moreover $h$-hyperdescent. We now consider the functor 
$$\Dd_{\mathrm{H}}(-):=\Ind\Dd^b(\mathrm{MHM}(-)):\Sch_\C^\op\to\mathrm{CAlg}(\mathrm{Pr}^L_\mathrm{St})$$ taking values in \emph{presentable} $\infty$-categories. By \cite[Appendix A]{SwannRealisation}, we see that both $\D_\mathrm{H}$ and $\D^b(\mathrm{MHM}(-))$ extend to non necessarily separated schemes using Zariski descent, together with all the operations.

\begin{rem}
  The above functor extends tautologically to a functor from diagram of schemes to diagram of symmetric monoidal stable $\infty$-categories, hence we can evaluate $\Dd_\mathrm{H}$ on a simplicial scheme to obtain a cosimplicial diagram of $\infty$-categories. Using the fact that it is easy (see \cite[Remark before 4.6]{MR1047415}) to compare the mixed Hodge structures on $\HH^n_{\mathrm{sing}}(X(\C),\Q)$ constructed by Deligne and Saito (under the equivalence $\mathrm{MHM}(\Spec(\C))\simeq\mathrm{MHS}^p_\C$) when $X$ is a closed subset of a smooth variety, $h$-hyperdescent of the derived category of mixed Hodge modules gives (using a resolution of singularities, that provides a $h$-hypercovering made out of smooth varieties) a simple proof that in fact the two mixed Hodge structures are the same for a general complex variety $X$. This result was known, but the proof is quite involved (see \cite{MR1741272}).
\end{rem}
\begin{rem}
  Everything we do in this article would probably hold more generally for arithmetic mixed Hodge modules over varieties defined over a subfield $k$ of $\C$. They are considered in \cite[Examples 1.8 (ii)]{saitoFormalismeMixedSheaves2006}, and our work in \cite{SwannRealisation}, thus the proofs of this article, would probably work \emph{verbatim} in this slightly more general context.
\end{rem}
\subsubsection{Hodge realisation of étale motives}
The main motivation for the $\infty$-categorical lifting of mixed Hodge modules was that this was the only obstruction for the existence of a realisation functor from Voevodsky motives that commutes with all the operations. We will deal here with the étale version, with rational coefficients.
Recall that the $\infty$-category of étale motives with rational coefficients over a scheme $X$ is defined (see \cite{MR3281141}) as a formula by 
$$\DM(X):= \Shv_{\et,\A^1}^\wedge(\mathrm{Sm}_X,\mathrm{Mod}_\Q)[\Q(1)^{\otimes -1}].$$ This means that to construct $\DM(X)$, one consider étale hypersheaves on smooth $X$-schemes with values in $\mathrm{Mod}_\Q\simeq\Dd(\Q)$ that are $\A^1$-invariant, and then one invert the object $\Q(1) = (M(\mathbb{P}^1_X)/M(\infty_X))[-2]$ for the tensor product, with $M$ the Yoneda embedding. The formula gives an universal property of the presentable symmetric monoidal $\infty$-category $\DM(X)$, as proven by Robalo in \cite{MR3281141}: any symmetric monoidal functor $F$ on smooth $X$-schemes with values in a rational presentably symmetric monoidal $\infty$-category $\Dd$ that satisfies étale hyperdescent, $\A^1$-invariance, and such that $F(\mathbb{P}^1_X)/F(\infty_X)$ is a tensor invertible object factors uniquely through the functor $M\colon \mathrm{Sm}_X\to\DM(X)$. By the work of Drew and Gallauer \cite{MR4560376} in fact this universality of $\DM(X)$ induces a universal property of the functor $\DM$ on schemes of finite type over some base $S$. Together with \cite[Theorem 4.4.25]{MR3971240}, taking $S = \Spec(\C)$ one obtains:

\begin{thm}[{\cite[Theorem 4.4]{SwannRealisation}}]
  There exists a Hodge realisation 
  $$\rho_{\mathrm{H}}\colon\DM\to\Dd_\mathrm{H}$$ on finite type $\C$-schemes that commutes with the $6$ operations. Moreover, the composition with the functor 
  $$\mathrm{rat}\colon\Dd_{\mathrm{H}}\to\Ind\Dd^b_c(-,\Q)$$ gives the Betti realisation.
\end{thm}
We will use this functor in the next section to compare mixed Hodge modules with Drew's approach, and at the end of this article to obtain the computation of the cohomology of a quotient stack with exhaustions in a easy way.

\subsection{Mixed Hodge modules of geometric origin}
\label{sectionMHMmotivic}
 In this section, we prove that Drew's approach in \cite{drewMotivicHodgeModules2018} to motivic Hodge modules gives a full subcategory of the derived category of mixed Hodge modules, stable under truncation, and thus has a t-structure.
\subsubsection{Motivic Hodge modules}
 As recalled above, we have a Hodge realisation 
 $$\rho_{\mathrm{H}}\colon\DM\to \Dd_{\mathrm{H}}$$ compatible with all the operations. This functor is $\mathrm{Mod}_\Q$-linear. The target $\Dd_{\mathrm{H}}$ is naturally valued in $\Dd_{\mathrm{H}}(\C)=\Ind\D^b(\mathrm{MHS}^p_\C)$-linear presentable $\infty$-categories, where $\mathrm{MHS}^p_\C$ is the abelian category of polarisable mixed Hodge structures over $\C$ (with rational coefficients). Thus the above realisation has a natural enrichment 
 $$\mathbf{\rho_{\mathrm{H}}}\colon \mathbf{{DM}}\to\Dd_{\mathrm{H}}$$ 
 where $\mathbf{DM}:=\DM\otimes_{\mathrm{Mod}_\Q}\Dd_{\mathrm{H}}(\C)$ is the $\Dd_{\mathrm{H}}(\C)$-linearisation of $\DM$. It can be computed as $$\mathbf{DM}(X) = \DM(X)\otimes_{\mathrm{Mod}_\Q} \Dd_{\mathrm{H}}(\C),$$ but can also be put inside the definition:
 $$\mathbf{DM}(X)\simeq \mathrm{Shv}_{\A^1,\et}^\wedge(\mathrm{Sm}_X,\Dd_{\mathrm{H}}(\C))[\Q(1)^{\otimes -1}]$$ (this follows from \cite[Corollary 2.24]{volpe2023operationstopology} which proves that tensoring can go inside sheaves, but also works for localisations such as $\A^1$-localisation. The commutation of tensoring with $\D_\mathrm{H}(\C)$ with inverting the Tate twist follows from the expression of the $\Q(1)$-tensor inverted category as a colimit (see \cite[1.4.12, 1.3.13, 1.6.3]{annalaMotivicSpectraUniversality2023}), and the preservation of colimits by the functor $-\otimes\Dd_\mathrm{H}(\C)$). Thus the $\infty$-category $\mathbf{DM}$ is the $\mathbb{P}^1$-stabilisation of the $\A^1$-invariant étale hypersheaves on $\mathrm{Sm}_X$ with values in $\Dd_{\mathrm{H}}(\C)$. In particular, it also affords the $6$ operations and the canonical functor $\DM\to\mathbf{DM}$ commutes with them.
For each finite type $\C$-scheme $X$, the functors 
$${\rho_{H,X}}\colon \DM(X)\to\Dd_{\mathrm{H}}(X)$$ 
and 
$${\rho_{\mathbf{H},X}}\colon \mathbf{{DM}}(X)\to\Dd_{\mathrm{H}}(X)$$ 
are colimit preserving symmetric monoidal functors, hence they have lax symmetric monoidal right adjoints $\rho^{H,X}_*$ and $\rho^{\mathbf{H},X}_*$ that create commutative algebras 

$$\hcal_X := \rho_*^{H,X}\Q_X \in\mathrm{CAlg}(\DM(X))$$ 
and $$\bm{\mathbf{\hcal}}_X := \rho_*^{\mathbf{H},X}\Q_X \in\mathrm{CAlg}(\mathbf{DM}(X))$$ in étale motives. Because the right adjoints commute with pushforwards, the counit maps induces algebra maps 
$$\hcal_{X/\C}:=\pi_X^*\hcal_\C\to\hcal_X$$ and $$\bm{\hcal}_{X/\C}:=\pi_X^*\bm{\hcal}_\C\to\bm{\hcal}_X,$$ where $\pi_X:X\to\Spec\C$ is the structural map. We will show below that those maps are in fact equivalences. 

\begin{defi}[Drew]
  The $\infty$-category of motivic Hodge modules over $X$ is 
  $$\mathrm{DH}(X):= \mathrm{Mod}_{\hcal_{X/\C}}(\DM(X)).$$
  There is an enriched version, that we will call the $\infty$-category of enriched mixed Hodge modules over $X$, defined as 
  $$\mathbf{DH}(X):=\mathrm{Mod}_{\bm{\hcal}_{X/\C}}(\mathbf{DM}(X)).$$
\end{defi}
An advantage of $\mathbf{DH}(X)$ when compared to $\mathrm{DH}(X)$ is that over a point, the $\infty$-category $\mathbf{DH}(X)$ is the whole derived category of graded polarisable mixed Hodge structures, whereas $\mathrm{DH}(X)$ only consists of mixed Hodge structures of geometric origin. However:
\begin{nota}
As the reader begins to guess, both situations, motivic Hodge modules and enriched motivic Hodge modules are parallel. Thus from now on except for the important results we will only deal with motivic Hodge modules, the proofs in the enriched case being the same. This is only to avoid doubling the size of this article, and the extensive use of bold.
\end{nota}

The functor $\DM$ is in fact valued in the $\infty$-category of $\DM(\C)$-linear presentable $\infty$-categories, and the construction $\mathrm{DH} = \mathrm{Mod}_{\hcal_{-/\C}}(\DM)$ can be rewritten 
$$\mathrm{DH} = \DM\otimes_{\DM(\C)}\mathrm{DH}(\C).$$
As $\DM(\C)$ and $\mathrm{DH}(\C)$ are rigid (this means that they are indization of small symmetric monoidal stable $\infty$-categories in which every object is dualisable: by \cite[Proposition 3.19]{MR3205601} they are compactly generated, by \cite[Proposition 2.2.27]{MR2423375} the motives of smooth projective varieties generate the compact objects, and by \cite[Théorème 2.2]{zbMATH02167776} the motives of smooth projective varieties are strongly dualisable), the $\infty$-functor 
\begin{equation}\label{infty2tens}-\otimes_{\DM(\C)}\mathrm{DH}(\C)\colon \mathrm{Pr}^L_{\DM(\C)}\to\mathrm{Pr}^L_{\mathrm{DH}(\C)}\end{equation} sending a $\DM(\C)$-linear presentable $\infty$-category $\ccal$ to $\ccal\otimes_{\DM(\C)}\mathrm{DH}(\C)\simeq\mathrm{Mod}_{\hcal_\C}(\ccal)$ has an $(\infty$-2)-categorical enhancement thanks to \cite[Section 4.4]{MR3607274}. In particular any adjunction 
$$\begin{tikzcd}
    \mathcal{C}
        \arrow[r, bend left = 25, "F"{name=D}]
        \arrow[r, leftarrow, bend right = 25, swap, "G"{name=C}]
          \arrow[d, from=D, to=C, phantom, "{\bot}"]
      & \mathcal{D}
\end{tikzcd}$$ between $\DM(\C)$-linear presentable $\infty$-categories, such that the right adjoint $G$ itself commutes with colimits and is $\DM(\C)$-linear (this is automatic if $F$ preserves compact objects by \cite[Proposition 4.9]{MR3607274} because $\DM(\C)$ and $\mathrm{DH}(\C)$ are rigid), the image under \cref{infty2tens} is again an internal $\mathrm{DH}(\C)$-adjunction, which means that $G\otimes\mathrm{DH}(\C)$ is the right adjoint to $F\otimes\mathrm{DH}(\C)$. Now it turns out that all properties of coefficients systems (\cite{MR4560376}) or motivic categories (\cite{MR3971240}) are properties of functors that are part of an internal adjunction, hence are preserved by ($\infty$-2)-functors. For example the property of $\A^1$-invariance of such system of $\infty$-categories is expressed as the counit $p_\sharp p^*\to\mathrm{Id}$ being an equivalence, where $p:\A^1_S\to S$ is the projection. Thus our functor 
$$\mathrm{DH}\colon \mathrm{Sch}_\C^\op \to\mathrm{Pr}^L_\mathrm{St}$$ is naturally a coefficient system in the sense of Drew and Gallauer, and therefore affords the 6 operations in a way that is compatible to the functor $\DM\to\mathrm{DH}$ (by \cite{MR2423375} and \cite{MR3971240}), and has $h$-descent (by \cite[Theorem 3.3.37]{MR3971240}). A proof of this result that does not use ($\infty$-$2$)-categories had been given by Drew in \cite[Theorem 8.10]{drewMotivicHodgeModules2018}.

\subsubsection{Embedding in mixed Hodge modules}
The Hodge realisation naturally factors as 
$$\DM\xrightarrow{\otimes \hcal_\C} \mathrm{DH} \xrightarrow{\underline{\rho_{\mathrm{H}}}}\Dd_{\mathrm{H}}$$ where all functors commute with the operations and all categories are compactly generated on schemes. Indeed one can see it in the following way:
$$\DM\to \mathrm{DH}\simeq\mathrm{Mod}_{\hcal_{(-)/\C}}(\DM)\xrightarrow{\rho_{\mathrm{H}}} \mathrm{Mod}_{\rho_{\mathrm{H}}(\hcal_{(-)/\C})}(\mathrm{DH})\to\mathrm{Mod}_{\Q}(\mathrm{DH})\xleftarrow{\sim}\mathrm{DH}$$
where the map $\rho_{\mathrm{H}}(\hcal_{(-)/\C})\to \Q$ is induced by the co-unit of the adjunction $(\rho_{\mathrm{H}},\rho^H_*)$.\\

The first observation is a consequence of the commutation with the operations:
\begin{lem}
  \label{ff}
  For each finite type $\C$-scheme, the functor $$\underline{\rho_{\mathrm{H}}}\colon\mathrm{DH}\to\Dd_{\mathrm{H}}$$ is fully faithful. Moreover the map 
  $$\hcal_{X/\C}\to\hcal_X$$ is an equivalence.
\end{lem}
\begin{proof}
  The proof is the same as the proof of \cite[Lemma 4.14]{SwannRealisation} and is originally due to Cisinski and Déglise. We recall it here. First over $\Spec\C$ the functor is fully faithful by \cite[Lemma 4.11]{drewMotivicHodgeModules2018}, because $\DM(\C)$ is rigid and $\rho_\mathrm{H}$ preserves compact objects. Over a general base it suffices to prove that the functor is fully faithful on compact objects $\mathrm{DH}^\omega$. Using that 
  \[\mathrm{Map}_{\Dd_\mathrm{H}(X)}(\underline{\rho_{\mathrm{H}}}(M),\underline{\rho_{\mathrm{H}}}(N))\simeq \mathrm{Map}_{\Dd_\mathrm{H}(\C)}(\underline{\rho_{\mathrm{H}}}(\Q),p_*\sHom(\underline{\rho_{\mathrm{H}}}(M),\underline{\rho_{\mathrm{H}}}(N))),\] where $p\colon X\to \Spec(\C)$ is the structural morphism, and that the same formula holds in $\mathrm{DH}$, we see that the lemma would follow from the commutation of $\underline{\rho_{\mathrm{H}}}$ with pushforwards and internal homomorphisms. Because the image of the compact preserving functor $-\otimes\hcal\colon\DM\to\mathrm{DH}$ generates $\mathrm{DH}$, it suffices even to prove that $\underline{\rho_{\mathrm{H}}}$ commutes with pushforwards and internal homomorphisms when restricted to the essential image of $-\otimes\hcal\colon (\DM)^\omega\to\mathrm{DH}^\omega$. This finishes the proof as $\rho_\mathrm{H}=\underline{\rho_{\mathrm{H}}}\circ(-\otimes \hcal)$ and $-\otimes\hcal$ commute with the 6 operations.
\end{proof}

Denote by $\HHh^n$ the cohomology functor for the standard t-structure.

\begin{defi}[Ayoub]
  Let $X$ be a finite type $\C$-scheme. 
  \begin{enumerate}\item We let $\mathrm{MHM}_\mathrm{geo}(X)$ (\emph{resp}. $\mathrm{MHM}_{hod}(X))$ be the full subcategory of $\Ind\mathrm{MHM}_\mathrm{std}(X)$ generated under kernels, cokernels, extensions and filtered colimits by objects of the form $\HHh^n(f_*\pi_Y^*K)(m)$ with $K\in\mathrm{Mod}_\Q$ (\emph{resp}. with $K\in\Dd_{\mathrm{H}}(\C)$), $f:Y\to X$ a proper morphism, $\pi_Y:Y\to\Spec\C$ the structural morphism and $n,m\in\Z$. These are Grothendieck abelian categories that we called Mixed Hodge modules of geometric origin (\emph{resp}. of Hodge origin).
    \item We let $\Dd_{H,\mathrm{geo}}(X)$ (\emph{resp}. $\Dd_{H,\mathrm{hod}}(X)$) be the full subcategory of complexes $K\in\Dd_{\mathrm{H}}(X)$ such that for all $n\in\Z$ we have $\HHh^n(K)\in\mathrm{MHM}_\mathrm{geo}(X)$ (\emph{resp}. we have $\HHh^n(K)\in\mathrm{MHM}_\mathrm{hod}(X)$). Those are stable $\infty$-categories on which the canonical t-structure on $\Dd_{\mathrm{H}}(X)=\Dd(\Ind\mathrm{MHM}_\mathrm{std}(X))$ restricts. Moreover they are stable under pullbacks by proper base change.
  \end{enumerate}
\end{defi}
\begin{lem}
  The functors 
  $$\underline{\rho_{\mathrm{H}}}\colon \mathrm{DH}\to\Dd_{\mathrm{H}}$$ and 
  $$\underline{\rho_{\mathbf{H}}}\colon\mathbf{DH}\to\Dd_{\mathrm{H}}$$ land in 
  $\Dd_{H,\mathrm{geo}}$ and $\Dd_{H,\mathrm{hod}}$.
\end{lem}
\begin{proof}
  This results from the commutation with the operations and the fact that $\mathrm{DH}(X)$ (\emph{resp.} $\mathbf{DH}(X)$) is generated under colimits by the $p_*\pi_X^*K(m)$ for $K\in\mathrm{Mod}_\Q$ (\emph{resp.} $K\in\Dd_{\mathrm{H}}(\C))$, thanks to \cite[Lemme 2.2.23]{MR2423375}.
\end{proof}
\begin{prop}
  \label{generices}
  Let $\eta$ be the generic point of an irreducible finite type $\C$-scheme $X$, and consider the functor 
  \begin{equation}\label{tututu}\colim_{\eta\in U}\mathrm{DH}(U)\to \colim_{\eta\in U}(\Dd_{H,\mathrm{geo}}(U))\end{equation} induced by $\underline{\rho_{\mathrm{H}}}$, where the colimit (taken in $\mathrm{Pr}^L$) runs over all the smooth open subsets $U$ of $X$ that contain $\eta$. It is an equivalence.
\end{prop}
\begin{proof}
  Note that by \cite[4.4.5.21, 5.5.7.8]{MR2522659}, the colimit is in fact computed in $\mathrm{Pr}^L_\omega$ the $\infty$-category of compactly generated $\infty$-categories, and we can rewrite it as 
  \[\Ind\Big(\colim_{\eta\in U}\mathrm{DH}(U)^\omega\to \colim_{\eta\in U}(\Dd_{H,\mathrm{geo}}(U))^\omega\Big)\] where now the colimit inside the $\Ind(-)$ is taken in $\catinfty$. By \cite[Proposition 2.1]{haine2025fullyfaithfulfunctorspushouts} the functor between compact objects is fully faithful, using \Cref{ff}. It remains to show that it is essentially surjective. It suffices to reach compact objects.

  Denote by 
  $$\dcal:=\colim_{\eta\in U}(\Dd_{H,\mathrm{geo}}(U)^\omega)$$ the colimit of the compact objects, computed in $\catinfty$.  
  
  For each $U$, the compact objects of $\Dd_{H,\mathrm{geo}}(U)$ have a standard t-structure and its heart is the category defined in the same way as $\mathrm{MHM}_\mathrm{geo}^\mathrm{std}(U)$ , but allowing only direct factors instead of all filtered colimits. As all transitions in the diagram are t-exact, the colimit $\dcal$ of the compact objects (which are the compact objects of the colimit) has a bounded t-structure. Moreover, 
  the category $\dcal$ is generated under finite colimits, finite limits, extensions and truncations by images in the colimit of objects of the form $g_*\pi_Z^*K(n)$ where $g\colon Z\to U$ is a proper morphism and $n\in\Z$ is an integer.

  By continuity of motives \cite[Lemma 5.1 (ii)]{MR4319065} and \cite[Theorem 4.8.5.11]{lurieHigherAlgebra2022}, the left hand side of \Cref{tututu} is canonically equivalent to $\mathrm{DH}(\eta):=\mathrm{Mod}_{u^*\hcal_\C}(\DM(\eta))$, where $u:\eta\to\Spec \C$ is the structural morphism. Moreover, by \cite[Lemma 2.2.27]{MR2423375}, the $\infty$-category $\mathrm{DH}(\eta)$ is generated under colimits by the $f_*\pi_Y^*K(n)$ where $f:Y\to\eta$ is a smooth and projective morphism, $n\in\Z$ and $K\in\mathrm{Mod}_\Q$. By spreading out (\cite[8.10.5(xiii)]{zbMATH03232548}, \cite[17.7.8(ii)]{zbMATH03245973}) and proper base change, such an object is the restriction to $\eta$ of a $g_*\pi_{Z}^*K(n)$ with $g:Z\to U$ a smooth projective morphism, and $U$ an open subset of $X$. 

  Now consider an object of the form $g_*\pi_Z^*K(n)$ in $\dcal$, where $g\colon Z\to U$ is a proper morphism and $n\in\Z$ is an integer. This object is the image by \Cref{tututu} of $(g_*\pi_Z^*K(n))_\eta \in\mathrm{DH}(\eta)$. By the above paragraph, there exists an open subset $V$ of $U$, a finite category $I$ and a functor $F\colon I\to \mathrm{DH}(V)$ such that for each $i\in I$, the object $F(i)$ is of the form $(f_i)_*\pi_{Z_i}^*K(n_i)$ for $f_i\colon Z_i\to U$ a \emph{smooth and projective} morphism and $n_i\in\Z$, and such that in $\mathrm{DH}(\eta)$, we have 
  \[ (g_*\pi_Z^*K(n))_\eta \simeq (\colim_I F(i))_\eta .\]
  In particular, in $\dcal$, the object $g_*\pi_Z^*K(n)$ is isomorphic to the image in the colimit over all open subsets, of the finite colimit $\colim_I F'(i)$, where $F'$ is the composition of the functor $F$ with \Cref{tututu}, which has the same formula for $F'(i)$.

  This gives that the category $\dcal$ is generated under finite colimits, finite limits, direct factors, extensions and truncations by images in the colimit of objects of the form $g_*\pi_Z^*K(n)$ with $g:Z\to U$ smooth and projective. We may also assume $U$ to be smooth.

 It turns out that to prove that the functor of the proposition is essentially surjective, it suffices to check that it reaches all objects of the heart of the compact objects. Thanks to \cite[Lemma 1.6.22]{ayoubAnabelianPresentationMotivic2022}, we see that it suffices to show that all subquotients of the image in the colimit of all $\HHh^n(f_*\pi_Y^*K)$, for $K\in\mathrm{Mod}_\Q$ compact and $f:Y\to U$ projective and smooth, are in the image. 
   Now by dévissage we can assume $K$ to be pure, so that by the conservation of weights under pushforwards by proper maps, and because $U$ is smooth hence $Y$ is also smooth so that $\pi_Y^*K$ is pure, the object $f_*\pi_Y^*K$ is pure.
   Using the decomposition theorem for pure complexes, we see that there is a decomposition \[f_*\pi_Y^*K\simeq \bigoplus_n\HHp^n(f_*\pi_Y^*K)[-n].\] In particular, each $\HHp^n(f_*\pi_Y^*K)$ lies in the image. 
   As $f$ is smooth and proper and $\pi_Y^*K$ is a dualisable object coming from $\Spec(\C)$, we see that (as in \cite[Theorem 2.2]{zbMATH02167776}, the dual is $f_*\pi_Y^*K^\vee(d)[2d]$ where $d$ is the relative dimension of $f$) the complex $f_*\pi_Y^*K$ is dualisable, so that $\HHp^n(f_*\pi_Y^*K)$ are local systems. Because duality datum consists of finitely many maps, and because compact objects in the colimit are dualisable, we see that any subquotient of the image of $\HHp^n(f_*\pi_Y^*K)$ in the colimit will lift as a subquotient in the category of local systems, of the restriction of $\HHp^n(f_*\pi_Y^*K)$ to a smaller open subset. Moreover, as $\HHp^n(f_*\pi_Y^*K)$ is lisse we have $\HHp^n(f_*\pi_Y^*K[d])[-d]\simeq \HHh^{n}(f_*\pi_Y^*K)$ where $d=\dim U$ (this is because over a smooth scheme, and for lisse complexes, the standard t-structure is the perverse t-structure, shifted by the dimension \cite[proof of Proposition 2.1.3]{MR0751966}). This proves that any subquotient of the image of $\HHh^n(f_*\pi_Y^*K)$ in the colimit is in fact a subquotient in perverse sheaves of a suitable restriction of $\HHp^n(f_*\pi_Y^*K[d])$ to an open subset of $U$. As this object is a pure perverse sheaf, any subquotient is in fact a direct factor, thus is in the image, and the proof is finished.
\end{proof}
This theorem is the same as Ayoub's theorem for Betti sheaves, and is proved the same way.
\begin{thm}
  \label{drewcomp}
  Let $X$ be a finite type $\C$-scheme.
  The functors 
  $$\underline{\rho_{\mathrm{H}}}\colon \mathrm{DH}(X)\to\Dd_{H,\mathrm{geo}}(X)$$ and 
  $$\underline{\rho_\mathbf{H}}\colon \mathbf{DH}(X)\to\Dd_{H,\mathrm{hod}}(X)$$  are equivalences of categories. 
\end{thm}
\begin{proof}
  This is now a simple Noetherian induction: if $K\in\Dd_{H,\mathrm{geo}}(X)$ is compact, we may find a nonempty irreducible open subset $U$ of $X$ with generic point $\eta$ and \Cref{generices} ensures that up to reducing the size of $U$, the restriction $K_{\mid U}$ of $K$ to $U$ is in the image, but the localisation sequence $$j_!K_{\mid U}\to K\to i_*i^*K$$ with $j:U\to X$ and $i:Z\to X$ the complement, together with the full faithfulness and a Noetherian induction, give that $K$ is in the image.
\end{proof}

This proves that all desiderata of Drew in \cite[Desiderata 1.1]{drewMotivicHodgeModules2018} are fulfilled.

\begin{rem}
  In a forthcoming work with Raphaël Ruimy, we use this point of view of motivic Hodge modules to construct categories of mixed Hodge modules of geometric origin that have $\Z$-linear coefficients.
\end{rem}

  Although what we have done ensures that the goals of Drew in \cite[Desiderata 1.1]{drewMotivicHodgeModules2018} are achieved, it may still seem unsatisfactory: another motivation for Drew's work was probably to give an alternative construction of mixed Hodge modules. Here our proof uses mixed Hodge modules hence it is not quite right. It should be possible to prove that motivic Hodge modules afford a t-structure without comparing them to M. Saito's mixed Hodge modules. 

  We give here a weak reason for which one could possibly define the t-structure only by geometric means (a t-structure on the generic points is enough to have a t-structure for every variety by gluing, see \cite[Theorem 3.1.4]{MR3347995}):

  \begin{prop}
    Let $K=\C(X)$ be the fraction field of an integral complex algebraic variety. Then the t-structure on 
    $$\mathrm{DH}(K/\C):= \colim_{U}\mathrm{DH}(U),$$ where the colimit runs over nonempty open subsets of $X$, is characterised by the fact that 
    $\mathrm{DH}(K/\C)^{\leqslant 0}$ is the smallest subcategory of $\mathrm{DH}(K/\C)$ stable under colimits, twists and extensions that contains the objects 
    $f_\sharp\Q_X$ for $f:X\to \Spec K$ smooth and affine.

  \end{prop}
  \begin{proof}
    We have that $\mathrm{DH}(K/\C)$ is compactly generated by $\colim_U\mathrm{DH}_c(U)$ where this time the colimit is taken in $\catinfty$. Thus we know that $\mathrm{DH}(K/\C)^{\leqslant 0}$ is the prestable $\infty$-category 
    $$\mathrm{DH}(K/\C)^{\leqslant 0}\simeq\Ind\big(\colim_U \mathrm{DH}_c(U)^{\leqslant 0}\big).$$
    In particular, as the heart of each $\mathrm{DH}(U)$ is generated under kernels, cokernels, extensions, twists and filtered colimits by objects of the form $\HHh^n(f_*\Q_Y)$ with $f\colon Y\to U$ proper, the proper base change theorem ensures that the heart of $\mathrm{DH}(K/\C)$ is generated under the same operations by the $\HHh^n(f_*\Q_Y)$ with $f:Y\to \Spec K$ proper and $n\in \N$. In fact we have seen in the proof of \Cref{generices} that we can even assume that $f$ is smooth and projective. Thus the prestable $\infty$-category $\mathrm{DH}(K/\C)^{\leqslant 0}$ is generated under colimits, extensions and twists by objects of the form $\HHh^n(f_*\Q)$ for $f\colon Y\to \Spec K$ smooth and projective. Moreover as this is true over a small open subset $U$ of $X$, we have a decomposition in $\mathrm{DH}(K/\C)$ of the form 
    $$f_*\Q_Y \simeq \bigoplus_n \HHh^n(f_*\Q_Y)[-n].$$
    This implies that $\mathrm{DH}(K/\C)^{\leqslant 0}$ is generated by the $f_*\Q_Y[2\dim Y]$ for $f:Y\to \Spec K$ projective and smooth. By covering $Y$ with smooth and affine schemes $(U_i)_{i\leqslant n}$, if one denote by $U_J = \cap_{i\in J}U_J$ for $J\subset \{1,\dots,n\}$, as we have a colimit diagram 
    $$ \begin{tikzcd}
        \dots
            \arrow[r, leftarrow, shift left = 0.5em]
            \arrow[r, shift left = 0.25em]
            \arrow[r, leftarrow]
            \arrow[r, shift right = 0.25em]
            \arrow[r, leftarrow, shift right = 0.5em]
          &  \bigoplus_{\mid J\mid = 2}f^{U_J}_!\Q_{U_J}[2d]
              \arrow[r, leftarrow, shift left = 0.25em]
              \arrow[r]
              \arrow[r, leftarrow, shift right = 0.25em]
            &  \bigoplus_{\mid J\mid = 1}f^{U_J}_!\Q_{U_J}[2d]
                \arrow[r]
              & f_*\Q_Y[2d]
    \end{tikzcd} $$
    by Zariski descent, with $f^{U_J}:U_J\to Y\to \Spec K$ and $d=\dim Y$, we see that $\mathrm{DH}(K/\C)^{\leqslant 0}$ is generated by the $f_!\Q_X[2d]$ for $f:X\to\Spec K$ smooth and affine of dimension $d$. As $f_\sharp\Q_X \simeq f_!\Q_X(d)[2d]$, this finishes the proof.

  \end{proof}

\subsection{Extension of the derived category of mixed Hodge modules to algebraic stacks}
\label{sectionStacks}

\subsubsection{Extensions and operations}
We have at hand a $h$-hypersheaf $$\Dd_\mathrm{H}\colon\mathrm{Sch}_\C^\op\to\mathrm{CAlg}(\mathrm{Pr}^L).$$
The functor $\Dd_{\mathrm{H}}$ extends (as a right Kan extension) canonically to a $h$-hypersheaf all Artin stacks over $\C$. More explicitly, given a presentation $\pi:X\to\mathfrak{X}$ of an algebraic stack $\Xfrak$ , that is a smooth surjection with $X$ a scheme, and assuming that the diagonal of $\mathfrak{X}$ is representable by schemes, and denoting by $$X^{n/\mathfrak{X}}:= X\times_\mathfrak{X}X\times_\mathfrak{X}\cdots X $$ the $n$-th fold of $X$ over $\mathfrak{X}$ (this is a scheme as the diagonal of $\mathfrak{X}$ is representable), we have a limit diagram in $\mathrm{Pr}^L$

\begin{equation}
  \label{defDesc}
  \begin{tikzcd}
    \Dd_{\mathrm{H}}(\mathfrak{X})
          \arrow[r,"\pi^*"]
        &  \Dd_{\mathrm{H}}(X)
            \arrow[r, shift left = 0.25em]
            \arrow[r, shift right = 0.25em]
          &  \Dd_{\mathrm{H}}(X\times_\mathfrak{X}X)
              \arrow[r, shift left = 0.5em]
              \arrow[r]
              \arrow[r, shift right = 0.5em]
            &  \dots
  \end{tikzcd}.
\end{equation}
If the diagonal of $\mathfrak{X}$ is not representable by schemes, one has to first define $\D_\mathrm{H}$ for algebraic spaces using the same formula as above.

Using Liu and Zheng gluing technique (\cite{liuEnhancedSixOperations2017}) as Khan in \cite[Appendix A.]{khanVirtualFundamentalClasses2019}, one can prove that the extension of $\Dd_{\mathrm{H}}$ to (higher) Artin stacks still have the six operations. Unless mentioned otherwise, all stacks and schemes considered will be locally of finite type over $\C$. This subsection is more or less book keeping of the work of Liu, Zheng and Khan.  More precisely, we have:

\begin{thm}[Liu-Zheng, Khan]
  \label{opera}
  \begin{enumerate}
    \item For every Artin stack $\mathfrak{X}$ there is a closed symmetric monoidal structure on $\Dd_{\mathrm{H}}(\mathfrak{X})$.
    \item For any morphism $f:\mathfrak{Y}\to\mathfrak{X}$ there is an adjunction $$ \begin{tikzcd}
        \Dd_{\mathrm{H}}(\mathfrak{X})
            \arrow[r, bend left = 25, "f^*"{name=D}]
            \arrow[r, leftarrow, bend right = 25, swap, "f_*"{name=C}]
              \arrow[d, from=D, to=C, phantom, "{\bot}"]
          & \Dd_{\mathrm{H}}(\mathfrak{Y})
    \end{tikzcd}$$ with $f^*$ a symmetric monoidal functor.
    \item For any locally of finite type morphism of Artin stacks $f:\mathfrak{X}\to\mathfrak{Y}$ there is an adjunction (functorial in $f$)
    $$\begin{tikzcd}
        \Dd_{\mathrm{H}}(\mathfrak{X})
            \arrow[r, bend left = 25, "f_!"{name=D}]
            \arrow[r, leftarrow, bend right = 25, swap, "f^!"{name=C}]
              \arrow[d, from=D, to=C, phantom, "{\bot}"]
          & \Dd_{\mathrm{H}}(\mathfrak{Y})
    \end{tikzcd}.$$
    \item The operations $f_!$ satisfy base change and projection formula against $g^*$, and the operations $f^!$ satisfy base change against $g_*$. If $f$ is representable by Deligne-Mumford stacks, then there is a natural transformation $\alpha_f:f_!\to f_*$, which is an isomorphism if $f$ is proper and representable by algebraic spaces (for a strenghtening, see \Cref{better4}).
    \item For any closed immersion $i:\mathfrak{Z}\to\mathfrak{X}$ of Artin stacks with $j:\mathfrak{U}\to\mathfrak{X}$ the inclusion of the open complement, we have a pullback diagram 
    $$\begin{tikzcd}
        \Dd_{\mathrm{H}}(\mathfrak{Z}) \arrow[r,"i_*"] \arrow[d] 
          & \Dd_{\mathrm{H}}(\mathfrak{X}) \arrow[d,"j^*"] \\
        * \arrow[r]
          & \Dd_{\mathrm{H}}(\mathfrak{U})
    \end{tikzcd}.$$

  \item For any Artin stack $\mathfrak{X}$, the functor 
  $$\pi^*:\Dd_{\mathrm{H}}(\mathfrak{X}) \to\Dd_{\mathrm{H}}(\A^1_\mathfrak{X})$$ is fully faithful.
  \item For a smooth morphism $f:\mathfrak{Y}\to\mathfrak{X}$ of relative dimension $d$, we have a purity isomorphism 
  $$\mathfrak{p}_f:  f^*(d)[2d]\simeq f^!.$$
  \item The $!$-functoriality $\Dd_{\mathrm{H}}(-)^!$ is also an étale hypersheaf.
\end{enumerate}
\end{thm}
\begin{proof}
  The proofs of \cite[Appendix A]{khanVirtualFundamentalClasses2019} hold for any motivic coefficient system, so that 1,2,3 and 4 are \cite[Theorem A.5]{khanVirtualFundamentalClasses2019}, 5. is \cite[Theorem A.9]{khanVirtualFundamentalClasses2019}, 6. is \cite[Proposition A.10]{khanVirtualFundamentalClasses2019} and 7. is \cite[Theorem A.13]{khanVirtualFundamentalClasses2019}. The last point 8. follows from \cite[Proposition 4.3.5]{liuEnhancedSixOperations2017}.

\end{proof}

\begin{rem}
All results above hold more generally for higher Artin stacks, except for $4.$ where one need $f$ to be $0$-truncated on top of being proper for $f_!$ to be isomorphic to $f_*$. 
\end{rem}

\begin{rem}
  One could state the same result for mixed Hodge modules of geometric or Hodge origin to stacks, and all the results we stated above and below also hold in this more restrictive setting. Note that because the construction is a right Kan extension from the category of correspondences of schemes to the category of correspondences of algebraic stacks, the functor $\mathrm{DH}\to\mathrm{D}_\mathrm{H}$ will automatically commute with the left adjoints $f^*$, $f_!$ and $\otimes$, but the commutation with the right adjoints is unclear. The commutation with internal homomorphisms and with $f_*$ for $f$ a representable morphism can be checked on a smooth atlas, but the cases of $f_*$ and $f^!$ for a general non-representable morphism seems subtle. See \Cref{propHB} for a result in that direction. 
\end{rem}

Now, as $\Dd^b(\mathrm{MHM}(-))$ also satisfies $h$-descent (in fact, étale descent is sufficient by \cite[Corollary 3.4.7]{khan2025lisseextensionsweaves}), we can also right Kan extend it and have the same formula as \cref{defDesc}. In particular as limits in $\catinfty$ are computed term wise, if one denote by $\Dd^b_{\mathrm{H},c}(-)$ the extension of $\Dd^b(\mathrm{MHM}(-))$ to (higher) Artin stacks, we have a fully faithful natural transformation 
$$\iota\colon \Dd^b_{\mathrm{H},c}(-)\to \Dd_{\mathrm{H}}(-).$$
The question whether for a general morphism $f$ of Artin stacks, the operations $f_*$ and $f_!$ preserve constructibility is a bit subtle. Of course, it is false in general that $f_*\Q$ is cohomologically bounded.

For a finite type $\C$-scheme $X$, the category $\Dd_{\mathrm{H}}(X)=\Ind\Dd^b(\mathrm{MHM}(X))$ admits a t-structure by \cite[Lemma C.2.4.3]{lurieSpectralAlgebraicGeometry}.  The t-structure on $\Dd_{\mathrm{H}}(X)$ is right complete and compatible with filtered colimits. We will use several times that a limit of stable $\infty$-categories with t-structures and t-exact transition functors is endowed with a canonical t-structure (\cite[Lemma 3.2.18]{MR4061978}).

\begin{prop}
  \label{tstructureperv}
  Let $\mathfrak{X}$ be an Artin stack locally of finite type  over $\C$. The $\infty$-category $\Dd_{\mathrm{H}}(\mathfrak{X})$ admits a t-structure such that for any smooth presentation $\pi:X\to\mathfrak{X}$ the conservative functor $\pi^*$ is t-exact up to a shift. Moreover, the t-structure is right complete and compatible with filtered colimits. It restricts to a t-structure on $\Dd^b_{\mathrm{H},c}(-)$.
\end{prop}
\begin{proof}
  We have to first deal with the case $\mathfrak{X}$ is an algebraic space. In this case, the map $\pi$ can be chosen étale, and as the diagonal of $\mathfrak{X}$ is representable by schemes, the limit diagram 
  \[\Dd_{\mathrm{H}}(\Xfrak)\to\lim_\Delta(X^{n/X})\] has t-exact transitions whence the limit $\Dd_{\mathrm{H}}(\Xfrak)$ has a perverse t-structure. The left separation and compatibility with colimits can be checked after applying $\pi^*$, hence hold. The right completeness is checked exactly as in the case of stacks, that we deal with below.

  We can assume that $\mathfrak{X}$ is a connected algebraic stack. Choose a presentation $\pi:X\to \Xfrak$ which is smooth of relative dimension $d$ for some integer $d$. Then all the projections $$p_n:X\times_\Xfrak X \times_\Xfrak \cdots \times_\Xfrak X \to X\times_\Xfrak  \cdots \times_\Xfrak X$$ are also smooth of relative dimension $d$, so that each $p_n^*[d]$ are t-exact for the perverse t-structure. As \cref{defDesc} can be done with the shifts $p_n^*[d]$, this creates a t-structure on the limit of the $\check{\text{C}}$ech nerve of $\pi$ such that $\pi^*$ is t-exact. It does not depend on $\pi$ as one can see by taking another presentation $p:Y\to\Xfrak$ and then $Y\times_\Xfrak X\to \Xfrak$. 
  The t-structure obviously restricts to $\Dd^b_{\mathrm{H},c}(\Xfrak)$. For the right completeness, we want to show that the natural functor 
  $$\Dd_{\mathrm{H}}(\Xfrak)\to \lim\Big(\cdots \xrightarrow{\tau^{\geqslant 1}} \Dd_{\mathrm{H}}^{\geqslant 1}(\Xfrak) \xrightarrow{\tau^{\geqslant 2}}\Dd_{\mathrm{H}}^{\geqslant 2}(\Xfrak)\xrightarrow{\tau^{\geqslant 3}}\Dd_{\mathrm{H}}^{\geqslant 3}(\Xfrak)\xrightarrow{\tau^{\geqslant 4}}\cdots\Big)$$ is an equivalence. As (shifts of) pullbacks by smooth maps are t-exact, this can be checked locally on $\Xfrak$, hence holds.
\end{proof}

\begin{defi}
  Let $\Xfrak$ be an Artin stack locally of finite type  over $\C$. The $\infty$-category of cohomologically constructible mixed Hodge modules over $\Xfrak$ is 
  $$\Dd_{\mathrm{H},c}(\Xfrak) := \{K\in\Dd_{\mathrm{H}}(\Xfrak)\mid\forall n\in \Z, \HH^n(K)\in\Dd^b_{\mathrm{H},c}(\Xfrak)\}.$$
\end{defi}
By definition the t-structure restricts to $\Dd_{\mathrm{H},c}(\Xfrak)$, and the heart is the same as the heart of $\Dd^b_{\mathrm{H},c}(\Xfrak)$.

For a finite type $\C$-scheme $X$, the stable $\infty$-category $\Dd^b(\mathrm{MHM}(X))$ also have a \emph{standard} t-structure (the t-structure for which the functor $\D^b(\mathrm{MHM}(X))\to\D^b_c(X^\mathrm{an},\Q)$ is t-exact when the target is endowed with the t-structure whose heart is the abelian category of constructible sheaves), and in \cite{SwannRealisation} (where the t-structure had the unfortunate name `constructible') we proved that it is the derived category of the constructible heart. In fact, all the considerations above about the perverse t-structure are true for the constructible t-structure. We mention them because the standard t-structure will happen to be handy when proving the constructibility of the operations.

\begin{prop}
  Let $\mathfrak{X}$ be an Artin stack over $\C$. The $\infty$-category $\Dd_{\mathrm{H}}(\mathfrak{X})$ admits a standard t-structure such that for any presentation $\pi:X\to\mathfrak{X}$ the conservative functor $\pi^*$ is t-exact if we endow $\Dd_{\mathrm{H}}(X)$ with the t-structure induced by the canonical t-structure on $\Dd^b(\mathrm{MHM}_\mathrm{std})$. Moreover, the t-structure is right complete and compatible with filtered colimits. It restricts to a t-structure on $\Dd^b_{\mathrm{H},c}(-)$. All pullbacks by all morphisms of Artin stacks are t-exact for this t-structure.
\end{prop}
\begin{proof}
  The proof is exactly the same as for \Cref{tstructureperv}, except that we do not need to shift the pullback functors.
\end{proof}

We will denote by $\HHh^n$ the cohomology objects for the standard t-structure.

\begin{prop}
  Let $\Xfrak$ be an Artin stack locally of finite type over $\C$ and let $K\in\Dd_{\mathrm{H}}(\Xfrak)$. Then $K\in \Dd_{\mathrm{H},c}(\Xfrak)$ if and only if for all $n\in\Z$ the object $\HHh^n(K)$ is in $\Dd^b_{\mathrm{H},c}(\Xfrak)$.
\end{prop}
\begin{proof}
  Both conditions on $K$ are local on $\Xfrak$, hence we can assume that $\Xfrak$ is a finite type $\C$-scheme $X$. If $K$ is bounded the result is trivial, and if not, we can reduce to the bounded case by noting the following: $\tau^{p\leqslant n}$ preserves $\Dd_{\mathrm{H}}(X)^{\mathrm{std}\leqslant n}$, 
  $\tau^{\mathrm{std}\geqslant n}$ preserves  $\Dd_{\mathrm{H}}(X)^{p\geqslant 0}$, 
  the restriction of $\tau^{p\geqslant n}$ to $\Dd_{\mathrm{H}}(X)^{\mathrm{std}\geqslant n}$ is the identity and the restriction of  $\tau^{\mathrm{std}\leqslant n}$ to $\Dd_{\mathrm{H}}(X)^{p\leqslant n}$ is the identity. Here $\tau^{\mathrm{std}}$ and $\tau^{p}$ are the truncation functors for the standard and perverse t-structures, respectively.
\end{proof}

\begin{thm}[Liu-Zheng]
  \label{thm_cosntr}
  Let $f:\Yfrak\to\Xfrak$ be a morphism of Artin stacks locally of finite type over $\C$. Then we have the following:
  \begin{enumerate}
    \item The $\infty$-category $\Dd_{\mathrm{H},c}^-(\Xfrak)$ is stable under tensor products, and $f^*$ restricts to a functor 
      $$f^*\colon \Dd_{\mathrm{H},c}(\Xfrak)\to\Dd_{\mathrm{H},c}(\Yfrak).$$
    \item The functor $f_*$ restricts to a functor 
      $$f_*\colon \Dd_{\mathrm{H},c}^+(\Yfrak)\to \Dd_{\mathrm{H},c}^+(\Yfrak) $$ and even to 
      $$f_*\colon \Dd_{\mathrm{H},c}(\Yfrak)\to \Dd_{\mathrm{H},c}(\Yfrak) $$ if $f$ is representable by algebraic spaces.
    \item The functor $f_!$ restricts to a functor 
    $$f_!\colon \Dd_{\mathrm{H},c}^-(\Yfrak)\to \Dd_{\mathrm{H},c}^-(\Yfrak) $$ and even to 
      $$f_!\colon \Dd_{\mathrm{H},c}(\Yfrak)\to \Dd_{\mathrm{H},c}(\Yfrak) $$ if $f$ is representable by algebraic spaces.
    \item The functor $f^!$ restricts to a functor 
    $$f^!\colon \Dd_{\mathrm{H},c}(\Xfrak)\to\Dd_{\mathrm{H},c}(\Yfrak).$$
    \item The internal Hom functor restricts to a functor 
    $$\Dd_{\mathrm{H},c}^-(\Xfrak)^\op\times \Dd_{\mathrm{H},c}^+(\Xfrak)\to\Dd_{\mathrm{H},c}^+(\Xfrak).$$
  \end{enumerate}
\end{thm}
\begin{proof}
  The proof is essentially the same as the proofs of \cite[6.4.4 and 6.4.5]{liuEnhancedSixOperations2017}. The point 1. is easy. We show how to obtain 2. by hand, for example:
  Take $\pi:X\to\Xfrak$ a presentation of $\Xfrak$ and form the pullback 
  $$
  \begin{tikzcd}
      \mathfrak{Z} \arrow[r,"g"] \ar[d,"q"]
          \arrow[dr, phantom, very near start, "{ \lrcorner }"]
        & X \arrow[d,"\pi"] \\
      \Yfrak \arrow[r,"f"]
        & \Xfrak
  \end{tikzcd}
  .$$ Then to show that each $\HHh^n(f_!K)$ are constructible, we may check it locally, whence as $\pi^*$ is t-exact it suffices to show that $\HHh^n(\pi^*f_!K)$ is constructible. By base change, it suffices to show that $\HHh^n(g_!q^*K)$ is constructible: we reduced to the case where $\Xfrak$ is a scheme. Now, choose a presentation $r:Z\to \Zfrak$ of $\Zfrak$. Descent for $!$-functoriality implies that $q^*K = \colim_\Delta (g_n)_!r_n^!q^*K$ where $(g_n)$ is the composition of the $\check{\text{C}}$ech  nerve $(r_n)$ of $r$ with $g$.   
  The spectral sequence induced by this geometric realisation reads 
  $$E_1^{p,q}=\HHh^{-q}((g_p)_!r_p^!q^*K)\Rightarrow \HHh^{p-q}(g_!q^*K).$$
  Now the each $r_p^!q^*K$ is cohomologically constructible by purity (7. of \Cref{opera}) and each $g_p$ is a morphism of algebraic spaces,  of relative dimension the sum of the relative dimensions of $g$ and of $r$. In particular, all $g_p$ have the same cohomological amplitude thus our spectral sequence is concentrated in a shifted quadrant: it vanishes if $p<0$ and if $q$ is smaller than some bound depending on the cohomological amplitude of $g_!$ and of $K$. This implies that the spectral sequence converges and we have that $\HH^n(g_!q^*K)$ is constructible for all $n$. In the case $f$ is representable by algebraic spaces, we can reduce to the case of schemes where it follows from finite cohomological amplitude. The case of $f_*$ is similar.
\end{proof}

\begin{rem}
  We may have to use the same extension to stack for others systems of coefficients such as analytic sheaves or étale motives. Because the functor between those and mixed Hodge modules commute with all the operations on schemes, the extension to stacks will commute with $\otimes$, $*$-pullbacks, $!$-pushforwards (by definition) and $\sHom(-,-)$, $*$-pushforwards and $!$-pullbacks for representable morphisms, by smooth base change so that it suffices to check this on an atlas (see \Cref{propHB}). This also applies to the Hodge realisation functor $\DM\to\D_\mathrm{H}$, and thus implies the following corollary.
\end{rem}

\begin{cor}
  Let $X$ be an complex algebraic variety and let $G$ be an algebraic group acting on $X$. Then if $X$ is smooth there exists a cycle class map 
  $$\mathrm{CH}^i_G(X)\to \mathrm{Hom}_{\Dd_\mathrm{H}(X/G)}(\Q,\Q(i)[2i])$$
  functorial in $X$, where the left hand side is the equivariant Chow group constructed by Edidin and Graham in \cite{MR1614555}. For a possibly singular algebraic one has to use Borel-Moore homology and we obtain a cycle class 
  \[\mathrm{CH}_i(X/G)\to \mathrm{Hom}_{\Dd_\mathrm{H}(X/G)}(\Q(i)[2i],\pi_{[X/G]}^!\Q(0))\] where $\pi_{[X/G]}\colon[X/G]\to\Spec(\C)$ is the structural morphism and the left-hand side is defined as a by Edidin and Graham in \cite[Section 5.3]{MR1614555}.
\end{cor}
\begin{proof}
  Once the left-hand side has been identified with a $\Hom$-group in $\mathrm{DM}$ by \cite[Example 12.17]{khanEquivariantGeneralizedCohomology2024} (or \cite[Corollary 6.5]{khanEquivariantGeneralizedCohomology2024} for the singular case), this is a consequence of the functoriality of the realisation functor.
\end{proof}

\subsubsection{Duality}
We want a working Verdier duality on algebraic stacks locally of finite type over the complex numbers. 

\begin{defi}
Let $\Xfrak$ be an algebraic stack locally of finite type over the complex numbers. The dualising object on $\Xfrak$ is the object $\omega_\Xfrak:=\pi_{\Xfrak}^!\Q(0)\in\Dd_{\mathrm{H}}(\Xfrak)$ with $\pi_{\Xfrak}\colon \Xfrak\to\Spec \C$. The Verdier duality functor is $$\mathbb{D}_{\Xfrak}:=\sHom(-,\omega_{\Xfrak}).$$
\end{defi}

If follows directly from the projection formula between $g_!$ and $g^*$, for $g$ a morphism of algebraic stacks, that there are natural equivalences 
$$g^!\circ \mathbb{D}_{\Xfrak}\simeq \mathbb{D}_{\Yfrak}\circ g^*$$ and 
$$g_*\circ\mathbb{D}_{\Yfrak}\simeq\mathbb{D}_{\Xfrak}\circ g_!$$
for any morphism $g:\Yfrak\to \Xfrak$ between algebraic stacks locally of finite type over $\C$.

\begin{lem}
  \label{Ddualliss}
  Let $g:\Yfrak\to\Xfrak$ be a \emph{smooth} morphism between algebraic stacks locally of finite type over $\C$. Then there is also a natural equivalence 
  $$g^*\circ \mathbb{D}_{\Xfrak}\simeq \mathbb{D}_{\Yfrak}\circ g^!$$ of functors 
  $\Dd_{\mathrm{H}}(\Xfrak)\to\Dd_{\mathrm{H}}(\Yfrak)$.
\end{lem}
\begin{proof}
  The smooth projection formula implies that the canonical map 
  \[g^*\sHom(-,\omega_X)\to \sHom(g^*(-),g^*\omega_\Xfrak)\] is an equivalence. Now, using the purity isomorphism twice $g^*\simeq g^!(-d)[-2d]$ we see that the right-hand side is equivalent to $\sHom(g^!(-),g^!\omega_\Xfrak)$, and because $g^!\omega_\Xfrak\simeq \omega_\Yfrak$, the proof is finished.
\end{proof}
\begin{prop}
  \label{dualityDBc}
  Verdier duality is an anti auto-equivalence when restricted to $\Dd_{H}^b(\Xfrak)$. It swaps $!$ and $*$ when they preserve $\Dd^b_{\mathrm{H},c}$ and it is perverse t-exact.
\end{prop}
\begin{proof}
  First note that duality preserves $\Dd^b_{\mathrm{H},c}$ because this can be checked locally.
  There is a canonical map $$\mathrm{Id}\to \mathbb{D}_{\Xfrak}\circ\mathbb{D}_{\Xfrak}.$$ For a given $M\in\Dd^b_{\mathrm{H},c}(\Xfrak)$ and a presentation $g:\Xfrak\to X$, it suffices to check that the map 
  $$g^*M\to g^*(\mathbb{D}_{\Xfrak}\circ\mathbb{D}_{\Xfrak}(M))$$  is an equivalence. But now by \Cref{Ddualliss} the above map is an equivalence because it is the map $$g^*M\to \mathbb{D}_X\circ\mathbb{D}_X(g^*M).$$ The remaining assertions follow from the autoduality over schemes (\cite[Theorem 0.1]{MR1047415}).
\end{proof}
\begin{cor}
  Verdier duality in fact extends to an equivalence of $\Dd_{\mathrm{H},c}^+(\Xfrak)\xrightarrow{\sim}\Dd_{\mathrm{H},c}^-(\Xfrak)^\op$. It swaps $!$ and $*$.
\end{cor}
\begin{proof}
  By the previous assertion, the functor $\mathbb{D}_{\Xfrak}$ is perverse t-exact, and an isomorphism on the heart. As the perverse t-structure is left separated when restricted to $\Dd_{\mathrm{H},c}^+(\Xfrak)$, this implies the claim.
\end{proof}

We now give an amelioration of \Cref{opera}(4) whose proof in the $\ell$-adic setting is due to Olsson \cite[Corollary 5.17]{MR3344762}, see also \cite[Theorem A.7]{khanVirtualFundamentalClasses2019}:
\begin{prop}
  \label{better4}
  Let $f:\Yfrak\to\Xfrak$ be a proper morphism represented by Deligne-Mumford stacks. Then the canonical morphism $f_!\to f_*$ is an isomorphism on $\Dd^+_{\mathrm{H},c}(\Yfrak)$. 
\end{prop}
\begin{proof}
  By \Cref{opera} we know the result if $f$ is representable by algebraic spaces. 
  Note that if $f$ is a proper morphism of algebraic stacks, then the formation of $f_*$ is compatible with base change on $\D^+_\mathrm{H,c}$. Indeed, using \cite[Theorem 1.1]{MR2183251} the proof is the same as \cite[Théorème 18.5.1]{MR1771927}: one uses the cohomological descent spectral sequence to reduce to the case of representable morphisms. Thus, by choosing a smooth atlas of the target of $f$, we can reduce to the case where $\Xfrak$ is a scheme and $\Yfrak$ is a Deligne-Mumford stack. Now by \cite[Théorème 16.6]{MR1771927}, there is a finite covering $Y\to \Yfrak$ with $Y$ a scheme, thus \cite[Theorem A.7.]{khanVirtualFundamentalClasses2019} applies.

\end{proof}
\begin{prop}
  \label{propHB}
  All result above about the extension of the operations work for the category $\Dd_B:= \Ind\Dd^b_c((-)^\mathrm{an},\Q)$ of Ind-constructible complexes, the forgetful functor $\mathrm{rat}\colon\Dd_{\mathrm{H}}\to\Dd_B$ extends to schemes, and commutes with the following functors: 
  \begin{enumerate}
    \item $-\otimes-$, $\sHom(-,-)$, and Verdier duality $\mathbb{D}_\Xfrak$.
    \item $f^*$ for any morphism.
    \item $f_!$ for any finite type morphism.
    \item $f_*$ for any proper morphism representable by Deligne-Mumford stacks.
  \end{enumerate}
  When restricted to $\D_\mathrm{H,c}^+$, it also commutes with:
  \begin{enumerate}
    \item Any pushforward $f_*$ for $f$ a finite type morphism.
    \item Any exceptional pullback $f^!$ for $f$ is a finite type morphism.
  \end{enumerate}
\end{prop}
\begin{proof}
    As both functors on stacks are defined using the right Kan extension on the category of correspondences (see the proof of \cite[Proposition A.5.16]{mann2022padic6functorformalismrigidanalytic}), it is clear that the map $\Dd_\mathrm{H}\to\Dd_B$ commutes with tensor products, $*$-pullbacks and $!$-pushforwards. Point 4.  is a direct consequence of 3. and \Cref{better4}.
    For $\sHom(-,-)$ and Verdier duality one can reduce to schemes directly using that $\mathrm{rat}$ commutes with twists and that both functors $\sHom$ and $\mathbb{D}_\mathfrak{X}$ commute with $\pi^*$ for $\pi$ a smooth morphism, up to a twist and a shift (by purity).

  Now on constructibles, Verdier duality is an equivalence an swaps $*$ and $!$'s, thus points $1$ and $2$ follow from $2$ and $3$ of the above list.

\end{proof}
\subsubsection{Weights}
In this subsection we introduce weights for mixed Hodge modules on stacks. We will work only with stack that are of finite type over $\C$ in this section. The abelian category of mixed Hodge modules over a scheme has a functorial and exact \emph{weight filtration}. Using ideas of the PhD of Sophie Morel, for each integer $w\in\Z$ we can define a t-structure ``$\omega\leqslant w$'' on $\D^b_\mathrm{H,c}(X)$, that we call the \emph{weight-t-structure} (\Cref{weight-t-struct}). It has the particularity that both its right and left sides are stable subcategories of $\D^b_\mathrm{H,c}(X)$, so that the heart is zero, but the truncations functors are exact functors. This provides the canonical extension of the weight filtration to the bounded derived category: we obtain exact functors $\tau^{(\omega\leqslant w)\leqslant 0}=:\omega_{\leqslant w}\colon \D^b_\mathrm{H,c}(X)\to \D^b_\mathrm{H,c}(X)$ that are perverse t-exact, and whose restriction to the perverse heart gives back the weight filtration of mixed Hodge modules. By indization we obtain the weight filtration on the Ind-category. For this t-structure, all pullbacks by smooth morphisms are t-exact, so that we can extend the weight filtration to any algebraic stack by descent (\Cref{weight-filtr-stacks}). In particular, the perverse heart of $\D_\mathrm{H}(\Xfrak)$ is endowed with an exact weight filtration. This gives the possibility to define \emph{punctual weights} (we chose this name because this is the correct analogue of punctual weights in the $\ell$-adic setting) in \Cref{defi:weights}.

However, this does not quite exactly gives a weight structure \emph{à la Bondarko}. Indeed, the weight filtration may fail to have sufficiently many orthogonality properties, so that the property that a pure perverse object is semi-simple may not be satisfied (see \Cref{wstructFail}). However, if $\Xfrak$ is an algebraic stack \emph{with affine stabilisers}, then the objects that are of punctual weight $0$ are indeed the heart of a weight structure, as we show in \Cref{wstruct}.

\begin{constr}
\label{weight-t-struct}
  Let $w\in\Z$ and $X$ be a finite type $\C$-scheme. Recall that following \cite{MR2350050} we can construct a weight-t-structure on $\Dd^b_{\mathrm{H},c}(X)$ by setting 
$$\Dd^b_{\mathrm{H},c}(X)^{\omega\leqslant w}:=\{K\in \Dd^b_{\mathrm{H},c}(X)\mid \forall n\in\Z, \forall m>w,\ \mathrm{gr}_m^W(\HHp^n(K))=0\}$$
and 
$$\Dd^b_{\mathrm{H},c}(X)^{\omega\geqslant w}:=\{K\in \Dd^b_{\mathrm{H},c}(X)\mid \forall n\in\Z, \forall m<w,\ \mathrm{gr}_m^W(\HHp^n(K))=0\},$$
where the $\mathrm{gr}_m^W$ are the graded pieces of the weight filtration on mixed Hodge modules (\cite[Proposition 1.5]{MR1042805}).
The pair $(\Dd^b_{\mathrm{H},c}(X)^{\omega\leqslant w},\Dd^b_{\mathrm{H},c}(X)^{\omega\geqslant w+1})$ form a t-structure (of trivial heart) on $\Dd^b_{\mathrm{H},c}(X)$, such that the truncations functors $\omega_{\leqslant w}$ and $\omega_{\geqslant w+1}$ are exact functors, t-exact for the perverse t-structure (\cite[Proposition 3.1.1.]{MR2350050}). They give a filtration $(\omega_{\leqslant w}K)_w$ on each complex $K$, that gives back the weight filtration on the heart. For any smooth map $f:Y\to X$ of finite type $\C$-schemes, and any integer $d$, the functor $f^*[d]$ is t-exact for the weight t-structures (indeed, if $d$ is the relative dimension of $f$ we see that $f^*[d]$ is perverse t-exact and preserves weights, thus is t-exact for the weight-t-structure. The shift functor being weight-t-exact this proves the claim). This induces a weight t-structure on $\Dd_{\mathrm{H}}(X)$ by indization and it has a similar description. 
\end{constr}
\begin{prop}
  \label{weight-filtr-stacks}
  Let $\Xfrak$ be an Artin stack of finite type over $\C$. Then the $\infty$-categories $\Dd_{\mathrm{H}}(\Xfrak)$ and $\Dd_{\mathrm{H},c}(\Xfrak)$ admit weight-t-structures such that for every smooth map $f:X\to \Xfrak$ with $X$ a finite type $\C$-scheme and every integer $d$, the functor $f^*$ is weight-t-exact. Moreover, the weight truncations functors are perverse t-exact. The inclusion functor $\Dd_{\mathrm{H},c}(\Xfrak)\to\Dd_{\mathrm{H}}(\Xfrak)$ is weight-t-exact.
\end{prop}
\begin{proof}
  Once again assume that $\Xfrak$ is connected, choose $\pi:X\to\Xfrak$ a presentation of $\Xfrak$, of relative dimension $d$. Then if $f_n$ is a part of the $\check{\text{C}}$ech  nerve of $\pi$, the functors $f_n^*[d]$ are t-exact for the weight t-structures, thus this induces a weight t-structure on the limit as in \cref{defDesc} (as in the proof of \Cref{tstructureperv} we have to first deal with algebraic spaces but everything works the same). Of course, it restricts to $\Dd^b_{\mathrm{H},c}(\Xfrak)$, and then to $\Dd_{\mathrm{H},c}(\Xfrak)$. The perverse t-exactness of the weight truncation follows from the fact that each $f_n^*[d]$ are also perverse t-exact so that this follows from the same property over schemes.
\end{proof}

\begin{defi}
  \label{defi:weights}
  Let $w\in \Z$.
  A object $K\in\Dd_{\mathrm{H}}(\Xfrak)$ is of \emph{punctual weights} $\leqslant w$ if for each $i\in\Z$, the map $\HHp^i(\omega_{\leqslant w+i}K)\to\HHp^i(K)$ is an equivalence (or equivalently, if $\HHp^i(\omega_{\geqslant w+i+1}K)=0$). We say that $K$ is of punctual weights $\geqslant w$ if its Verdier dual $\mathbb{D}_\Xfrak(K)$ is of punctual weights $\leqslant -w$.
  We shall denote by $\Dd_{\mathrm{H},(c)}^{?}(\Xfrak)_{\leqslant w}$ and $\Dd_{\mathrm{H},(c)}^{?}(\Xfrak)_{\geqslant w}$ the corresponding full subcategories, where $?\in\{\emptyset,b,+,-\}$. 
\end{defi}

\begin{rem}
  The above definition \emph{differs} from the weight t-structure we used to define weights.
\end{rem}

\begin{prop}
  \label{weightbasiccompati}
  Let $f:\Yfrak\to \Xfrak$ be a morphism of algebraic stacks of finite type over $\C$ and let $w,w'\in\Z$.
  \begin{enumerate}
    \item Verdier duality on $\Xfrak$ swaps $\Dd_{\mathrm{H},c}^+(\Xfrak)_{\leqslant w }$ and $\Dd_{\mathrm{H},c}^-(\Xfrak)_{\geqslant -w}$.
    \item The pullback $f^*$ sends $\Dd_{\mathrm{H},c}(\Xfrak)_{\leqslant w }$ to $\Dd_{\mathrm{H},c}(\Yfrak)_{\leqslant w }$ and the exceptional pullback $f^!$ sends 
    $\Dd_{\mathrm{H},c}(\Xfrak)_{\geqslant w}$ to $\Dd_{\mathrm{H},c}(\Yfrak)_{\geqslant w}$. 
    If $f$ is smooth, the pullback functor $f^*$ preserves weights.
    \item The tensor product restricts to 
    $$-\otimes- \colon \Dd_{\mathrm{H},c}^-(\Xfrak)_{\leqslant w} \times \Dd_{\mathrm{H},c}^-(\Xfrak)_{\leqslant w'}\to \Dd_{\mathrm{H},c}^-(\Xfrak)_{\leqslant w+w'}.$$
    \item The internal homomorphism functor restricts to 
      $$\sHom(-,-)\colon \Dd_{\mathrm{H},c}^-(\Xfrak)_{\leqslant w} \times \Dd_{\mathrm{H},c}^-(\Xfrak)_{\geqslant w'}\to \Dd_{\mathrm{H},c}^-(\Xfrak)_{\geqslant w'-w}.$$
  \end{enumerate}
\end{prop}

\begin{proof}
  By definition, if $\pi:X\to\Xfrak$ is a smooth presentation, the functor $\pi^*$ is conservative and detects weights. Thus, all results follow from the usual results over schemes (see \cite[Propositions 1.7 and 1.9]{MR1042805}).
\end{proof}

\begin{prop}
  \label{better310}
  Let $\Xfrak$ be an algebraic stack with affine stabilisers, of finite type over $\C$. Let $f:\Xfrak\to \Yfrak$ be a finite type morphism. If $K$ is an object of $\Dd_{\mathrm{H}}(\Xfrak)$ of punctual weights $\leqslant w$, then $f_!K$ has punctual weights $\leqslant w$.
  \end{prop}
  \begin{proof}
      By choosing a smooth presentation $\pi$ of $\Yfrak$, using proper base change, and the fact that $\pi^*[d]$ is conservative, perverse t-exact and weight-t-exact, it detects punctual weights, so we may assume that $\Yfrak=Y$ is a separated variety.
      First note that for any stack $\Xfrak$, we have the formula 
      \[\Dd_{\mathrm{H}}(\Xfrak)_{\leqslant w}=\{K\in \Dd_{\mathrm{H}}(\Xfrak)\mid \forall i\in\Z, \HHp^i(\omega^{w+i+1}K)=0\}.\]
      In particular, as the functors $\omega^{w+i+1}$ and $\HHp^i$ commute with filtered colimits, we see that $\Dd_{\mathrm{H}}(\Xfrak)_{\leqslant w}$ is stable under filtered colimits in $\Dd_{\mathrm{H}}(\Xfrak)$. 

      By \cite[Proposition 3.5.2 and Proposition 3.5.9]{MR1719823} and the fact that $\Xfrak$ is of finite type over $k$, there exists a finite stratification $\Xfrak_i$ of $\Xfrak$ by locally closed substacks such that each $\Xfrak_i\simeq [X_i/G_i]$ is a global quotient stack of a quasi-projective variety $X_i$ by a smooth connected algebraic group $G_i$ acting linearly on $X_i$. As $\Dd_{\mathrm{H}}(Y)_{\leqslant w}$ is closed under extensions, using the localisation triangle we see that we may assume that $\Xfrak\simeq [X/G]$ is a global quotient stack with $X$ quasi-projective. 

      By \cite[Remark 1.4]{MR1743244} we may find a specific increasing sequence $(U_n)$ of representations of $G$ (hence vector bundles on $BG$), called a Borel resolution, that helps compute the invariants of $[X/G]$ as follows: 
      Let $X\overset{G}{\times}U_n := [X/G]\times_{BG} U_n$, then for any $\A^1$-invariant étale sheaf $F\colon \mathrm{Sm}^\op_{[X/G]}\to\ccal$ on the category of schemes that are smooth over $[X/G]$ with values in an $\infty$-category $\ccal$ admitting all small limits, the canonical map 
      \[F^\lhd([X/G])\to \lim_n F^\lhd(X\overset{G}{\times}U_n)\] is an equivalence, where $F^\lhd$ is the right Kan extension of $F$ to algebraic stacks. This is \cite[Theorem 3.6]{khanEquivariantGeneralizedCohomology2024}. Note that by \cite[Remark 3.4]{khanEquivariantGeneralizedCohomology2024}, each $X\overset{G}{\times}U_n$ is a quasi-projective variety.

      Let $K\in \Dd_{\mathrm{H}}([X/G])_{\leqslant w}$. We apply the above to the functor $F$, opposite to the functor that sends a smooth map $g\colon Y\to[X/G]$, with $Y$ a variety, to $\pi^Y_!g^!K$. By $!$-descent for the étale topology and $\A^1$-invariance of $\D_\mathrm{H}$ this functor $F$ satisfies the conditions stated in the previous paragraph. Moreover, its right Kan extension to stacks over $[X/G]$ has the same formula, again by $!$-descent. In particular, we have that the map 
      \[\colim_n (fp_n)_!p_n^!K\to f_!K\] is an equivalence, where $p_n\colon X\times^G U_n\to [X/G]$ is the projection. Note that this is a filtered colimit, so that it suffices to prove that each $(fp_n)^!p_n^!K$ is of punctual weights $\leqslant w$. But as $p_n$ is smooth, the weights of $p_n^!K$ are the weights of $K$, and the map $fp_n$ is a map of schemes, so that we have reduced to the case of schemes, where this is true by \cite[Proposition 1.7]{MR1042805}.
 \end{proof}

  \begin{cor}
    \label{fshriekweights}
      Let $f:\Yfrak\to\Xfrak$ be a morphism of algebraic stacks of finite type over $\C$. Assume that $\Yfrak$ has affine stabilisers. Then if $K\in \Dd^+_{\mathrm{H},c}(\Yfrak)$ is of weights $\geqslant w$ the object $f_*K$ also has weights $\geqslant w$.
  \end{cor}
  \begin{proof}
    By duality, this follows from the case of $f_!$. 
  \end{proof}

  \begin{cor}
    \label{wstruct}
    Let $\Xfrak$ be an algebraic stack with affine stabilisers, of finite type over $\C$. Then there is a weight structure \emph{à la} Bondarko on $\Dd_{H}^b(\Xfrak)$ whose heart consist of objects of weight $0$ as in \Cref{defi:weights}. In particular, any pure object in $\Dd_\mathrm{H}^b(\Xfrak)$ is the sum of its cohomology objects, and any pure object of the  perverse heart is semi-simple.
  \end{cor}
  \begin{proof}
    We want to use \cite[Theorem 4.3.2]{MR2746283} to construct a weight structure from its heart of pure objects of weight 0. Thus we want to prove that pure objects of weight $0$ generated $\Dd^b_{\mathrm{H},c}(\Xfrak)$ as a thick subcategory, and that it is negative (for us, semi-simple will be enough). Thus if $K,L\in\Dd_{H}^b(\Xfrak)$ are pure of weight zero, we want to show that $\Hom_{\Dd_{H}^b(\Xfrak)}(K,L[1])$ vanishes. Let $f:\Xfrak\to\Spec\C$ be the structural morphism. As we have 
    $$\Hom_{\Dd_{H}^b(\Xfrak)}(K,L[1])\simeq \Hom_{\Dd_{\mathrm{H}}(\Spec\C)}(\Q_{\Spec \C},f_*\sHom(K,L)[1]),$$ by \Cref{fshriekweights} and the usual orthogonality in the derived category of mixed Hodge structures, it suffices to show that $\sHom(K,L)$ has weights $\geqslant 0$, which results from \Cref{weightbasiccompati}.

    Now objects of weight $0$ generate the $\infty$-category under finite colimits. Indeed by dévissage it suffices to show that objects of the perverse heart which are pure are in the $\infty$-category generated by the complexes of weight $0$, but now this is only a shift. The last sentence of the corollary is now formal, see \cite[Théorème 5.4.5]{MR0751966} and \cite[Théorème 5.3.8]{MR0751966}.
  \end{proof}

  \begin{rem}
    \label{wstructFail}
    The above \Cref{wstruct} is \emph{false} for general algebraic stacks, as shown by Sun (following an example that goes back to Drinfeld) in \cite{MR2972459}. Indeed one can show that the pushforward of the constant sheaf by the quotient map $\pi\colon \Spec(\C)\to\mathrm{B}E$ with $\mathrm{B}E$ the classifying space of an elliptic curve is not semi-simple although $\pi$ is smooth and proper.
  \end{rem}
\subsubsection{Nearby cycles}
Let $\Xfrak$ be an algebraic stack locally of finite type over $\C$ and let $f:\Xfrak\to\A^1_\C$ be a function. We have the following diagram:
\begin{equation}\label{setupNC}\begin{tikzcd}
	\Ufrak & \Xfrak & \Zfrak \\
	{\Gm_{,\C}} & {\A^1_\C} & {\{0\}}
	\arrow[from=1-1, to=2-1,"f_\eta"]
	\arrow["j", from=2-1, to=2-2]
	\arrow["i"', from=2-3, to=2-2]
	\arrow[from=1-3, to=2-3,"f_s"]
	\arrow["f"', from=1-2, to=2-2]
	\arrow["i"', from=1-3, to=1-2]
	\arrow["j", from=1-1, to=1-2]
\end{tikzcd}.\end{equation}
We wish to construct a nearby cycle functor 
$$\Psi_f\colon \Dd_{\mathrm{H}}(\Ufrak)\to\Dd_{\mathrm{H}}(\Zfrak)$$ with good properties. We begin with the unipotent part of the nearby cycle following \cite{cassCentralMotivesParahoric2024b}, whose authors give a wonderful interpretation of Ayoub's unipotent nearby cycle functor (as in \cite{MR2438151}). 

In \cite[Definition 2.22]{cassCentralMotivesParahoric2024b}, Cass, van den Hove and Scholbach define a $\Q$-linear stable $\infty$-category $\Nilp$ (denoted by $\Nilp_\Q$ in \emph{loc. cit.}) whose objects can be interpreted as pairs $(K,N)$ with $K\in\mathrm{Mod}_\Q^\Z$ a complex of graded $\Q$-vector spaces and $N:K(1)\to K$ is a locally nilpotent map of graded complexes, where $K(1):= K\otimes_\Q \Q(1)$ and $\Q(1)$ is the graded $\Q$-vector space placed in degree $-1$. This $\infty$-category is compactly generated by objects $(\Q(k),\Q(k+1)\xrightarrow{0}\Q(k))$, and if $(K,N)\in\Nilp$ is compact, its underlying complex $K$ is compact and the operator $N$ is nilpotent. Now for any stably presentable $\infty$-category $\ccal$ with and action of $\mathrm{Mod}_\Q^\Z$, on can define 
$$\Nilp\ccal := \ccal\otimes_{\mathrm{Mod}_\Q^\Z}\Nilp,$$ which has a similar description as pairs $(c,N)$ with $c\in\ccal$ and $N:c(1)\to c$, which is nilpotent if $c$ is compact. 

The relation of this construction with unipotent nearby cycles $\Upsilon_f$ comes from the desire to have a monodromy operator on $\Upsilon_f$, thus a lift of $\Upsilon$ as a functor 
$$\Dd_{\mathrm{H}}(\mathfrak{U})\to\Nilp\Dd_{\mathrm{H}}(\mathfrak{Z}).$$

This is made possible thanks to the following observation. 

\begin{prop}
  Let $p:\Gm_\C\to\Spec\C$ be the projection. Then the pushforward 
  $$p_*\colon \Dd_{HT}(\Gm)\to\Dd_{HT}(\C)$$ exhibits the $\infty$-category $\Dd_{HT}(\Gm)$ as $\Nilp\Dd_{HT}(\C)$. Here we have denoted by $\Dd_{HT}(X)$ the full subcategory of $\Dd_{\mathrm{H}}(X)$ of mixed Hodge Tate modules over $X$, that is the $\infty$-category generated under colimits and shifts by the $\Q(i)$ for $i\in\Z$.
\end{prop}
\begin{proof}
  This is the combination of \cite[Corollary 2.20 and Lemma 2.17]{cassCentralMotivesParahoric2024b}, that we reproduce for the reader's convenience. First we remark that the functor $p_*$ indeed preserves Hodge Tate objects because $p_*\Q(n)\simeq \Q(n)\oplus\Q(n+1)[1]$. Now, it is clear that $p_*$ preserves colimits as its left adjoint $p^*$ preserves compact objects, and it is also conservative because if $p_*M=0$ then for each $n\in\Z$ the mapping spectra $\Map_{\Dd_{HT}(\Gm_\C)}(\Q(n),M)\simeq \Map_{\Dd_{HT}(\C)}(\Q(n),p_*M)$ vanishes, so that $M=0$. Thus, by Barr-Beck's theorem, the functor $p_*$ upgrades to an equivalence 
  $$\Dd_{\mathrm{H}}(\Gm_\C)\xrightarrow{\sim} \mathrm{Mod}_{p_*\Q}(\Dd_{\mathrm{H}}(\C)).$$
  There is a unique colimit preserving symmetric monoidal functor $$i\colon\mathrm{Mod}_\Q^\Z\to \Dd_{HT}(\C)$$ that sends $\Q(1)$ to $\Q(1)$.
  The image under the Hodge realisation of \cite[Lemma 2.18]{cassCentralMotivesParahoric2024b} tells us that $p_*\Q$ is, as a commutative algebra object, the object $i(\Lambda)$ where $\Lambda$ is the split square zero extension $\Lambda:=\Q\oplus\Q(-1)[-1]$. Thus we have that 
  $$\Dd_{HT}(\Gm_\C)\simeq \mathrm{Mod}_{p_*\Q}(\Dd_{HT}(\C))\simeq \Dd_{HT}(\C)\otimes_{\mathrm{Mod}^\Z_\Q}\mathrm{Mod}_{\Lambda}(\mathrm{Mod}_\Q^\Z).$$ 
  and the equivalence $\mathrm{Mod}_{\Lambda}(\mathrm{Mod}_\Q^\Z)\simeq \Nilp$ of \cite[Lemma 2.9]{cassCentralMotivesParahoric2024b} finishes the proof.
\end{proof}

We come back to the setup \cref{setupNC}, and define, as \cite[Definition 3.1]{cassCentralMotivesParahoric2024b}:

\begin{defi}
  \label{defUnip}
  The unipotent nearby cycle functor is the composite 
  $$\Upsilon_f \colon \Dd_{\mathrm{H}}(\Ufrak)\to \mathrm{Mod}_{f_\eta^*p^*p_*\Q}(\Dd_{\mathrm{H}}(\Ufrak))\simeq \Nilp\Dd_{\mathrm{H}}(\Ufrak)\xrightarrow{i^*j_*} \Nilp\Dd_{\mathrm{H}}(\Zfrak)$$
  where the first functor is the way one can see any object $K\in\Dd_{\mathrm{H}}(\Ufrak)$ as a $f_\eta^*p^*p_*\Q$-module with the counit $f_\eta^*p^*p_*\Q\otimes K\to K$, and the second is expressing that $i^*j_*$ preserves the nilpotent structure.
\end{defi}

In fact as they show in \cite[Section 3.1]{cassCentralMotivesParahoric2024b}, there is a better way to look at the above definition.
Indeed, the construction of the operations for stacks in \Cref{opera} goes by extending the functor $\Dd_{\mathrm{H}}$ on schemes to a lax monoidal functor 
$$\Dd_{H}\colon\mathrm{Corr}_\C \to \mathrm{Pr}^L_\mathrm{St},$$ 
where $\mathrm{Corr}_\C$ is the $\infty$-category of correspondences of stacks over $\C$: its objects are derived algebraic stacks locally of finite type over $\C$ and its morphisms are informally diagrams 
$$Y\xleftarrow{g} Z\xrightarrow{f} X$$ of algebraic stacks locally of finite type over $\C$. Composition is given by forming pullback squares. The functor sends  $X$ to $\Dd_{\mathrm{H}}(X)$, and the above morphism to 
$$g_!f^*\colon \Dd_{\mathrm{H}}(Y)\to\Dd_{\mathrm{H}}(X).$$ In fact, the projection formulae even allow us to have a functor with values in $\mathrm{Mod}_{\Dd_{H}(\C)}(\mathrm{Pr}^L_\mathrm{St})$.

Now, unipotent nearby cycles can be upgraded to a highly coherent natural transformation $\Upsilon\colon \Dd_{H\eta}\to\Nilp\Dd_{Hs}$ of lax monoidal functors 
$$\mathrm{Corr}_{\A^1_\C}^{\mathrm{pr,sm}}\to\mathrm{Pr}_\mathrm{gr}^L,$$ where $\mathrm{Cor}_{\A^1_\C}^{\mathrm{pr,sm}}$ is the wide subcategory (so it has the same objects but fewer morphisms) of $\mathrm{Corr}_{\A^1_\C}=(\mathrm{Corr}_\C)_{/\A^1_\C}$ whose morphisms are the roofs $Y\xleftarrow{g} Z\xrightarrow{f}X$ over $\A^1$ with $g$ proper and $f$ smooth, and $\mathrm{Pr}_\mathrm{gr}^L:=\mathrm{Mod}_{\mathrm{Mod}_\Q^\Z}(\mathrm{Pr}^L_\mathrm{St})$ is the $\infty$-category of presentable $\infty$-categories with an action of graded complexes of vector spaces.

Here, $\Dd_{H\eta}$ is the restriction of the composition 
$$\Dd_{H\eta}\colon \mathrm{Corr}_{\A^1}\xrightarrow{-\times_{\A^1}\Gm} \mathrm{Corr}_{\Gm} \xrightarrow{\Dd_{\mathrm{H}}} \mathrm{Mod}_{\Dd_{HT}(\Gm_\C)}(\mathrm{Pr}^L_\mathrm{St})\xrightarrow{U}\mathrm{Mod}_{\Dd_{HT}(\C)}(\mathrm{Pr}^L_\mathrm{St})\to \mathrm{Pr}_\mathrm{gr}^L$$

with $U:\Dd_{HT}(\Gm)\to\Dd_{HT}(\C)$ the forgetful functor, and $\Dd_{Hs}$ is the restriction of the composition 
$$\Dd_{Hs}\colon\mathrm{Corr}_{\A^1_\C}\xrightarrow{-\times_{\A^1}\{0\}} \mathrm{Corr}_\C \xrightarrow{\Dd_{\mathrm{H}}} \mathrm{Pr}_\mathrm{gr}^L\xrightarrow{\Nilp}\mathrm{Pr}_\mathrm{gr}^L.$$

Theorem 3.2 in \emph{loc. cit.} is:
\begin{thm}[Cass, van den Hove, Scholbach]
  There is a natural transformation of lax monoidal functors $\mathrm{Corr}^{\mathrm{pr,sm}}_{\A^1_\C}\to\mathrm{Pr}^L_\mathrm{gr}$ 
  $$\Upsilon\colon\Dd_{H\eta}\to\Nilp\Dd_{Hs},$$ 
  whose evaluation at a stack $X/\A^1$ is the functor defined in \Cref{defUnip}
\end{thm}
\begin{proof}
  The proof is \emph{verbatim} the same as their proof in \cite[Section 3.3]{cassCentralMotivesParahoric2024b}. The only thing to check is that over stacks, where our categories are not compactly generated (but over schemes everything is compactly generated hence for example mixed Hodge Tate categories are rigid), all functors in the play indeed preserve colimits. Only $j_*$ could be a problem, but the preservation of colimits can be checked on a presentation and it thus follows from smooth base change. Also, $\Nilp\Dd_{Hs}$ stays a $h$-sheaf because $\Nilp$ is compactly generated, hence the functor $\otimes_{\mathrm{Mod}_\Q^\Z}\Nilp$ commutes with limits.
\end{proof}
We will denote by $\Upsilon_f\colon \Dd_{\mathrm{H}}(\Ufrak)\to\Dd_{\mathrm{H}}(\Zfrak)$ the functor obtained after applying $\Upsilon$ to $f:\Xfrak\to\A^1_\C$ and forgetting the monodromy locally nilpotent operator $N$.
The fact that $\Upsilon$ above is a lax monoidal natural transformation expresses the usual compatibilities of $\Upsilon_f$ with proper pushforward and smooth pullbacks, together with a natural transformation 
$$\Upsilon_{f_1}\boxtimes \Upsilon_{f_2}\to\Upsilon_{f\times g}$$ compatible with the nilpotent operators when one has two functions $f_i\colon \Xfrak_i\to\A^1_\C$. 

\begin{prop}
  Let $f:\Xfrak\to\A^1_\C$ be a function on a algebraic stack locally of finite type over $\C$. The functor $\Upsilon_f[-1]$ preserves $\Dd^b_{\mathrm{H},c}$, is perverse t-exact and commutes with the external tensor product.
\end{prop}
\begin{proof}
  The functor $$\mathrm{rat}\colon\Dd_{\mathrm{H}}\to\Dd_B$$ commutes with $\Upsilon_f$ by \cite[Lemma 3.18]{cassCentralMotivesParahoric2024b} (note that in the proof only is needed right adjointability for $j^*$ for $j$ an open immersion which is in particular representable, so that we may use \Cref{propHB}). On top of this, it detects constructibility and is conservative on constructible objects.
  Moreover, by taking a presentation $\pi:X\to\Xfrak$ of $\Xfrak$, it suffices to prove the proposition for $\Upsilon_f$ with $f:X\to\A^1_\C$ a function on a scheme. Thus the proof of the proposition reduces to the case of a function of a scheme, on analytic sheaves. The comparison \cite[Example 3.19]{cassCentralMotivesParahoric2024b} of $\Upsilon$ in this case, with Beilinson's unipotent nearby cycles functor \cite{MR0923134} imply the result, because for analytic sheaves the functor $\Upsilon_f[-1]$ is perverse t-exact and commutes with the external tensor product.
  
\end{proof}

\begin{prop}Let $f:\Xfrak\to\A^1_\C$ be a function on a algebraic stack locally of finite type over $\C$.
  There is an exact triangle 
  $$\Upsilon_f[-1]\xrightarrow{N}\Upsilon_f(-1)[-1]\to i^*j_*.$$ 
\end{prop}

\begin{proof}
  This is \cite[Proposition 3.9]{cassCentralMotivesParahoric2024b}, shifted by $-1$.
\end{proof}

\begin{prop}
  \label{compSaitoUsp}
  Let $X$ be a reduced and separated finite type $\C$-scheme and let $f:X\to\A^1_\C$ be a function. Then there is a natural isomorphism of functors 
  $$\Upsilon_f[-1]\simeq \psi_{f,1}$$ on the triangulated category $\Dd^b(\mathrm{MHM}(U))$, where $U = f^{-1}(\Gm)$ and $\psi_{f,1}$ is the unipotent nearby cycle defined by Saito in \cite{MR1047415}.
\end{prop}
\begin{proof}
  As both functors are t-exact it suffices to construct an isomorphism on $\mathrm{MHM}(U)$. By \cite[Proposition 1.3]{MR1047741} for $K\in\mathrm{MHM}(U)$ there is a canonical isomorphism 
  $$\colim_n\HHp^{-1}i^*j_*(K\otimes  f^*E_n)\xrightarrow{\sim} \psi_{f,1}K,$$
  where $E_n = \Q\oplus\cdots \Q(-n)\in\D_\mathrm{H}(\Gm)$.
  The proof of the comparison with Beilinson construction \cite[Example 3.19]{cassCentralMotivesParahoric2024b} gives in fact that 
  $$\Upsilon_f(K)\simeq \colim_n i^*j_*(K\otimes f^*E_n),$$ so that the compatibility with colimits of the t-structure gives the result.
\end{proof}

\begin{cor}
  On $\Dd^b_{\mathrm{H},c}$, the natural map $$ \Upsilon_f\circ \mathbb{D}_{\Ufrak}(-)\to \mathbb{D}_{\Zfrak}\circ\Upsilon_f(-)(1)[2].$$ is an equivalence.
\end{cor}
\begin{proof}
  The map is defined as follows:
  $\Upsilon_f$ is lax monoidal hence there is a canonical map 
  \[\Upsilon_f \circ \sHom(-,f^!_\eta\Q)\otimes \Upsilon_f(-)\to \Upsilon_f(\sHom(-,f^!_\eta\Q)\otimes (-))\] that we can compose with the evaluation map 
  \[\sHom(-,f^!_\eta\Q)\otimes (-)\to \Upsilon_f(f^!_\eta\Q)\] to obtain a map 
  \[\Upsilon_f \circ \sHom(-,f^!_\eta\Q)\otimes \Upsilon_f(-)\to f^!_\eta\Q.\]
  By adjunction, this provides a map 
  \[\Upsilon_f\circ\sHom(-,f^!_\eta\Q)\to \sHom(\Upsilon_f(-),\Upsilon_f(f^!_\eta\Q)).\]
  Now, the commutation of $\Upsilon$ with $!$-pushforwards implies that there is a exchange map \[\Upsilon_f\circ f_\eta^!\to f_\sigma^!\circ\Upsilon_{\mathrm{Id}},\] and by \cite[Lemma 3.12]{cassCentralMotivesParahoric2024b} (to which we apply the Hodge realisation), we have that $\Upsilon_\mathrm{Id}(\Q)\simeq \Q$.
  This gives a map 
  \[ \Upsilon_f\circ\sHom(-,f^!_\eta\Q)\to \sHom(\Upsilon_f(-),f^!_\sigma \Q)\] in which in the right hand side we recognise (because $f_\sigma$ is the structural map of $\Zfrak$) the functor $\mathbb{D}_\Zfrak\circ \Upsilon_f$. 
  The dualising object on $\Ufrak$ is $\omega_\Ufrak = f_\eta^!\pi_{\Gm}^!\Q \simeq f_\eta^!\Q(1)[2]$, with $\pi_\Gm$ the structural map of $\Gm$. Using that $\Upsilon_f$ commutes with Tate twists (this follows easily from the projection formula and the fact that it commutes with $p_!$, for $p\colon \Xfrak\times \mathbb{P}^1\to\Xfrak$), we finally have a map 
  \[\Upsilon_f\circ\mathbb{D}_\Ufrak(-) \to \mathbb{D}_\Zfrak\circ\Upsilon_f(-)(1)[2].\]
  This comparison map is an isomorphism on $\D^b_{\mathrm{H}}$ because this can be checked locally hence once finite type $\C$-schemes, where it has been proven by Saito (alternatively, one can apply the functor to analytic sheaves and use the result of Beilinson).
\end{proof}
Now, for each $n\in \N^*$ we may consider $\pi_n$ the elevation to the $n$-th power in $\A^1_\C$, and form the cartesian square
$$\begin{tikzcd}
    \Xfrak_n \arrow[r,"e_n"] \arrow[d,"f_n"]
        \arrow[dr, phantom, very near start, "{ \lrcorner }"]
      & \Xfrak \arrow[d,"f"] \\
    \A^1_\C \arrow[r,"\pi_n"]
      & \A^1_\C
\end{tikzcd}.$$
Denote by $i_n\colon f_n^{-1}(\{0\})\to f^{-1}(\{0\})$ be the nil-immersion obtained by restricting $e_n$ (the functor $(i_n)_*$ is an equivalence).
Then by functoriality of $\Upsilon$ obtain a system of functors $((i_n)_*\circ\Upsilon_{f_n}\circ e_n^*)_{n\in(\N^*)^\op}$, where we still denote by $e_n$ the restriction of $e_n$ to the inverse image of $\Gm$.
\begin{defi}
  The total nearby cycle functor is the functor 
  $$\Psi_f \colon \Dd_{\mathrm{H}}(\Ufrak)\to\Dd_{\mathrm{H}}(\Zfrak)$$ defined as 
  $$\Psi_f:=\colim_{n\in(\N^*)^\op}(i_n)_*\Upsilon_{f_n}\circ e_n^*.$$
\end{defi}
\begin{rem}
  This definition in the case of motives is due to Ayoub (\cite{MR2438151}), but its formal description as a colimit can be found in Preis' \cite{preisMotivicNearbyCycles2023}.
\end{rem}

\begin{prop}
  Let $f:X\to \A^1_\C$ be a function on a reduced and separated finite type $\C$-scheme. Then there is a natural equivalence of triangulated functors 
  $$\Psi_f[-1]\simeq \Psi_{f}^S$$ on $\Dd^b(\mathrm{MHM}(U))$, where $\Psi_f^S$ is Saito's nearby cycle functor. In particular, $\Psi_f$ preserves constructible objects.
\end{prop}
\begin{proof}
  First, because the t-structure is compatible with colimits and each $e_n^*$ are perverse t-exact, it is clear that $\Psi_f[-1]$ is perverse t-exact on $\Dd(\Ind\mathrm{MHM}(U))$. Thus, by dévissage it suffices to construct an isomorphism $\Psi_f[-1]\simeq \Psi_f^S$ of exact functors on the heart $\mathrm{MHM}(U)$. 

  By definition (\cite[Proposition 5.7]{saitoFormalismeMixedSheaves2006}) together with the comparison of \Cref{compSaitoUsp}, for $M\in\mathrm{MHM}(U)$ we have 
  $$\Psi_f^S(M)=\Upsilon_f(M)[-1]=\Upsilon_f(M\otimes f_n^*(\pi_n)_*\Q)[-1]$$ for $n$ divisible enough, in other terms, $$\Psi_f^S = \colim_n\Upsilon_f(-\otimes f_n^*(\pi_n)_*\Q)[-1].$$

  For each $n\in\N^*$ because $e_n$ is proper we have $(i_n)_*\Upsilon_{f_n}\simeq \Upsilon_f\circ (e_n)_*$ and now the proper projection formula and proper base change give $(e_n)_*\simeq (-\otimes f_n^*(\pi_n)_*\Q)$, finishing the proof.
\end{proof}
\begin{thm}
  \label{nearbcycles}
  Let $f:\Xfrak\to \A^1_\C$ be a function on an algebraic stack locally of finite type over $\C$.
  \begin{enumerate}
  \item The functor $$\Psi_f[-1]\colon \Dd_{\mathrm{H}}(\Ufrak)\to\Dd_{\mathrm{H}}(\Zfrak)$$
  preserves the full subcategory $\Dd^b_{\mathrm{H},c}$, is perverse t-exact and lax monoidal.
  \item Given another function $g:\Yfrak\to\A^1_\C$, the natural map 
  $$\Psi_f(-)\boxtimes \Psi_g(-)\to \Psi_{f\times g}(-\boxtimes -)$$ is an equivalence. 
  \item On $\Dd_{\mathrm{H},c}^+(\Ufrak)$ the natural transformation $$\Psi_f\circ\mathbb{D}_\Ufrak(-)\to\mathbb{D}_\Zfrak\circ \Psi_f(-)(1)[2]$$ is an equivalence.
  \end{enumerate}
\end{thm}
\begin{proof}
  Thanks to the same property for $\Upsilon$ the functor $\Psi$ commutes with smooth pullback, hence all those properties are local for the smooth topology and can be checked on a separated finite type $\C$-scheme, where they follow from the results on the underlying perverse sheaves and the comparison with Saito's functor. The last assertion follows from the result on the bounded category and t-exactness.
\end{proof}

We finish this section with the remark that because taking cones is functorial in the world of stable $\infty$-categories, we have a vanishing cycles functors for free:
\begin{defi}
  Let $f:\Xfrak\to \A^1_\C$ be a function on an algebraic stack locally of finite type over $\C$. We define the vanishing cycle functor 
  $$\Phi_f\colon \Dd_{\mathrm{H}}(\Xfrak)\to\Dd_{\mathrm{H}}(\Zfrak)$$ as the cone of the natural map $i^*\to \Psi_fj^*$.
  We thus have an exact triangle 
  $$ i^*\to \Psi_fj^*\xrightarrow{\mathrm{can}} \Phi_f.$$
\end{defi}
\begin{prop}
  The functor $$\Phi_f[-1]\colon \Dd_{\mathrm{H}}(\Xfrak)\to\Dd_{\mathrm{H}}(\Zfrak)$$ preserves $\Dd^b_{\mathrm{H},c}$, is perverse $t$-exact and commutes with Verdier duality up to a twist $-(1)$. 
\end{prop}
\begin{proof}
  This can be checked on a smooth atlas, thus on schemes where this holds for in constructible sheaves $\Dd^b_c(-,\Q)$.
\end{proof}
\subsubsection{Comparison with existing constructions}
In this section, we compare our construction to the construction of Achar \cite{acharEQUIVARIANTMIXEDHODGE} that dealt with equivariant mixed Hodge modules, as well as with Davison's \cite{davisonPurity2CalabiYauCategories2024} in which they constructed pushforwards of the constant object under a morphism from a stack.

In Achar's \cite{acharEQUIVARIANTMIXEDHODGE}, the construction goes as follows: $G$ is an affine algebraic group acting on a complex algebraic variety $X$, and he defines a triangulated category $\Dd_G(X)$ of $G$-equivariant mixed Hodge modules on $X$. This triangulated has a perverse t-structure, whose heart is indeed the abelian category of $G$-equivariant mixed Hodge modules. The definition is the following: for a special Grothendieck topology called the \emph{acyclic topology} in which covering are smooth $G$-equivariant morphism satisfying some cohomological vanishing (it is asked that for some $n$ and universally, $f$ satisfies that $\tau^{\leqslant n}f_*f^*$ is the identity when restricted to the abelian category of mixed Hodge modules), it turns out that the presheaf of triangulated categories 
$U\mapsto \mathrm{ho}\Dd^b(\mathrm{MHM}(X))$ is a sheaf, and as $X$ admits a covering in the acyclic topology by a $U$ on which $G$ acts freely (so that $[U/G]$ is a scheme), this sheaf canonically extends to $G$-equivariant varieties. 
\begin{prop}
Let $X$ be a $G$-variety. Then the homotopy category of $\Dd^b_{\mathrm{H},c}([X/G])$ coincides with Achar's category $\Dd_G^b(X)$.
\end{prop}
\begin{proof}
  First, if $G$ acts freely on $X$ (for the definition of a free action, see \cite[Definition 6.4]{acharEQUIVARIANTMIXEDHODGE}), then $[X/G]$ is a scheme and the result is trivial. Then exactly for any $n$-acyclic map $U\to X$ such that the action of $G$ on $U$ is free and $a\leqslant b$ integers such that $b-a<n$, the same proof as \cite[Lemma 8.1]{acharEQUIVARIANTMIXEDHODGE} implies that the functor 
  $$\Dd_{\mathrm{H}}^{[a,b]}([X/G])\to\Dd_{\mathrm{H}}^{[a,b]}([U/G])$$ is fully faithful. Thus, the functor 
  $$\mathrm{ho}\Dd^b_{\mathrm{H},c}([X/G])\to\Dd_G(X)$$ is fully faithful. It is not hard to check that the heart of $\Dd^b_{\mathrm{H},c}([X/G])$ is also the abelian category of $G$-equivariant mixed Hodge modules (because they satisfy descent), so that this finishes the proof.
\end{proof}
Similarly one can check that the 6 operations defined by Achar coincide with ours.\\

We now compare the pushforward we constructed with the one considered in \cite{davisonPurity2CalabiYauCategories2024}. Let $X$ be a smooth algebraic variety over $\C$, and let $G$ be an affine algebraic group acting on it. We are interested to the perverse cohomology groups of the object $p_!\Q$, where 
$$p\colon \Xfrak:=[X/G]\to \mcal$$ is a morphism of stacks, with $\mcal$ an algebraic variety. Here is how Davison and Meinhardt proceed: the construction is very similar to the construction of the compactly supported motive of a classifying space by Totaro in \cite{MR3548464}, and  our proof of the comparison is inspired from the proof of \cite[Proposition A.7]{MR4222601} by Hoskins and Pépin-Lehalleur and goes back to Borel. We will denote by $X/G$ the quotient stacks instead of $[X/G]$.

First, they choose an increasing family $V_1\subset V_2\subset\cdots \subset V_i\subset \cdots $ of representations of $G$, and a subsystem $U_1\subset U_2\subset\cdots \subset U_i\subset \cdots$ of representations on which $G$ acts freely (this can be done by choosing a closed embedding $G\subset \mathrm{GL}_r(\C)$ and then setting $V_i = \Hom_\C(\C^i,\C^r)$, and $U_i$ is the subset of surjective linear applications). Then $U_i/G$ is an algebraic variety, and in the case they deal with (they ask for the group $G$ to be \emph{special}), the quotient stack $(X\times U_i)/G$ is also a scheme by \cite[Proposition 23]{MR1614555} of Edidin and Graham. We denote by $\mathfrak{V}_i:=(V_i\times X)/G$, $\mathfrak{U}_i:=(U_i\times X)/G$. We have a commutative diagram: 
\[\begin{tikzcd}
	{\mathfrak{V}_i} & {\mathfrak{X}} & {\mathcal{M}} \\
	{\mathfrak{U}_i}
	\arrow["{b_i}", from=1-1, to=1-2]
	\arrow["{\iota_i}", from=2-1, to=1-1]
	\arrow["{a_i}"{description}, from=2-1, to=1-2]
	\arrow["p", from=1-2, to=1-3]
	\arrow["{p_i}"', from=2-1, to=1-3]
	\arrow["{q_i}"{description}, from=1-1, to=1-3, bend left=30]
\end{tikzcd}\]
Moreover the maps $a_i,b_i$ and $\iota_i$ are smooth. Davison and Meinhardt prove that for a given $n\in\Z$, the object 
$$\HHp^n((p_i)_!\Q_{\Ufrak_i}(-ir)[-2ir])$$ is independent of the choices for $i$ and $M$ large depending on $n$, and they denote by $\HHp^n(p_!\Q_{\Xfrak})$ the common value. 
We will show that indeed the canonical map $$\HHp^n((p_i)_!\Q_{\Ufrak_i}\{-ir\})\to\HHp^n(p_!\Q_{\Xfrak})$$ is an isomorphism in $\mathrm{MHM}(\mcal)$, where we have denoted by $(-)\{n\}:=(-)(n)[2n]$ for $n\in\Z$. In fact, we will show better, as the result (which is classical, and see \cite{khanEquivariantGeneralizedCohomology2024} for a vast generalisation, and \Cref{better310} for another use of this method) holds universally in motives:
\begin{prop}
  \label{calculpshriek}
  The natural map $$\colim_i(p_i)_!\Q_{\Ufrak_i}\{-ir\}\to p_!\Q_{\Xfrak}$$ is an equivalence in $\DM(\mcal,\Q)$, where $\DM(\mcal,\Q)$ is the $\infty$-category of étale motives with rational coefficients.
\end{prop}
\begin{proof}
  First note that the counit map 
  $$(b_i)_!(b_i)^!\Q_{\Xfrak}\to\Q_\Xfrak$$ is an equivalence by $\A^1$-invariance, because $b_i$ is a vector bundle on $\Xfrak$. As $b_i$ is smooth of relative dimension $ir$, purity gives that the canonical map 
  $$(q_i)_!\Q_{\mathfrak{V}_i}\{-ir\}\to p_!\Q_{\Xfrak}$$ is an equivalence for each $i\in\N$. Thus it suffices to prove that the map 
  $$\colim_i (a_i)_!\Q_{\mathfrak{U}_i}\{-ir\} \to\colim_i (b_i)_!\Q_{\mathfrak{V}_i}\{-ir\}$$ on $\Xfrak$ induced by the counits $(\iota_i)_!\iota_i^*\to\mathrm{Id}$ is an equivalence, because applying $p_!$ would produce the sought isomorphism. 
  
  Denote by $\pi:X\to\Xfrak$ the projection. Recall that $\pi^*$ is conservative, proper base change ensures that it suffices to deal with the analogous situation over $X$
  \[\begin{tikzcd}
    {{V}_i} & {{X}} \\
    {{U}_i}
    \arrow["{\beta_i}", from=1-1, to=1-2]
    \arrow["{j_i}", from=2-1, to=1-1]
    \arrow["{\alpha_i}"{description}, from=2-1, to=1-2]
  \end{tikzcd}\]
  where every thing has been pulled back to $X$, and the map is now 
  $$\colim_i (\alpha_i)_!\Q_{{U}_i}\{-ir\} \to\colim_i (\beta_i)_!\Q_{{V}_i}\{-ir\}.$$ Everything is now a scheme. The family of pullbacks $(x^*)_{x\in X}$ is conservative on $\DM(X,\Q)$ by localisation, thus by proper base change we reduce further to the case where $X = \Spec k$ is the spectrum of a field, and each $V_i$ is a smooth $k$-scheme (in fact, a vector space). Furthermore, the codimension of $W_i:= V_i\setminus U_i$ is $V_i$ goes to $\infty$ when $i\to\infty$, and the smoothness of $\alpha_i$ and $\beta_i$ ensure that 
  $$(\alpha_i)_!\alpha_i^*\Q\{-ir\} \simeq (\alpha_i)_\sharp\alpha_i^*\Q\simeq M(U_i)\in\DM(k,\Q),$$ and the same for $(\beta_i)_!\Q\{-ir\}=M(V_i)$, where for $Y$ a smooth $k$-scheme, the object $M(Y)$ is the motive of $Y$.
  We are looking at 
  $$\colim_i M(U_i)\to\colim_i M(V_i)$$ in $\DM(k,\Q)$, with $\mathrm{codim}_{V_i}(V_i\setminus U_i)\underset{i\to\infty}{\longrightarrow}\infty$. By \cite[Proposition 2.13]{MR4222601} this is an equivalence.
\end{proof}

\begin{cor}
In $\Dd_{\mathrm{H}}(\mcal)=\Ind\Dd^b(\mathrm{MHM}(\mcal))$ the natural map 
$$\colim_i(p_i)_!\Q_{\Ufrak_i}\{-ir\}\to p_!\Q_{\Xfrak}$$ is an equivalence. In particular, for each $n\in\N$, the natural maps 
$$\HHp^n((p_i)_!\Q_{\Ufrak_i})\{-ir\} \to \HH^p(p_!\Q_{\Xfrak})$$
are equivalences for $i$ big enough.
\end{cor}
\begin{proof}
  The Hodge realisation $$\rho_{\mathrm{H}}\colon\DM\to\Dd_{\mathrm{H}}$$ extends naturally to stacks, in a way that commutes with the operations and colimits, thus the first statement is just the Hodge realisation applied to \Cref{calculpshriek}. The second statement follows from the fact that the t-structure on $\Dd_{\mathrm{H}}(\mcal)$ is compatible with filtered colimits, and that each $\HHp^n(p_!\Q_\Xfrak)$ is constructible thus Noetherian.
\end{proof}

\subsubsection*{Acknowledgements} 
I am grateful to Pierre Descombes and Tasuki Kinjo for encouraging me to write this article and for giving me pertinent feedback. Tasuki Kinjo very kindly pointed out references I had missed. I am indebted to the anonymous referees for many comments and for suggesting a simplification of the proof of \Cref{better310}. I thank Joseph Ayoub, Mark de Cataldo, Pavel Safranov and Jakob Scholbach for comments on earlier versions of this article. I would also like to thank Frédéric Déglise, Sophie Morel, Raphaël Ruimy and Luca Terenzi for useful discussions.

This article was written during my PhD whose fundings are given by the ÉNS de Lyon. I was also funded by the ANR HQDIAG.

Swann Tubach \url{swann.tubach@ens-lyon.fr}

E.N.S Lyon, UMPA, 

46 Allée d'Italie, 69364 Lyon Cedex 07, France
\end{document}